\documentclass[11pt,namelimits,sumlimits,a4paper]{amsart}
\usepackage{cite}
\usepackage{comment}

\usepackage{amssymb,amsmath}
\usepackage[mathscr]{eucal}
\usepackage{slashed}

\usepackage[T1]{fontenc}
\usepackage[UKenglish]{babel}
\usepackage{amsfonts}
\usepackage{fancyhdr}
\usepackage{graphicx}
\usepackage{amsmath}
\usepackage{amsthm}
\usepackage{amsmath,amscd}
\usepackage{latexsym}
\usepackage{cite}
\usepackage{amssymb,amsmath}
\usepackage{tensor}

\usepackage{comment}
\usepackage[mathscr]{eucal}
\usepackage{slashed}
\usepackage{times,psfrag}
\usepackage{mathtools}
\usepackage{hyperref}
\usepackage{multirow}
\usepackage{longtable}
\usepackage{bm}
\usepackage[all]{xy}
\usepackage{fancyhdr}
\usepackage{float}
\usepackage{esint}
\usepackage{layout}
\usepackage{enumitem}
\usepackage{lmodern}
\usepackage[rgb]{xcolor}
\usepackage{tikz}
\usetikzlibrary{shapes,arrows}
\usetikzlibrary{matrix,arrows}
\usepackage{dsfont}
\usepackage{tikz-cd}
\usepackage{stmaryrd}

\numberwithin{equation}{section}
\newtheorem{theorem}[equation]{Theorem}
\newtheorem{lemma}[equation]{Lemma}
\newtheorem{prop}[equation]{Proposition}

\newtheorem*{theorem*}{Theorem}

\theoremstyle{definition}
\newtheorem{definition}[equation]{Definition}
\newtheorem{example}[equation]{Example}

\theoremstyle{remark}
\newtheorem{remark}[equation]{Remark}
\newtheorem*{remark*}{Remark}

\newcommand{\ie}{\emph{i.e.} }
\newcommand{\eg}{\emph{e.g.} }
\newcommand{\cf}{\emph{cf.} }

\newcommand{\beq}{\begin{equation}}
\newcommand{\eeq}{\end{equation}}
\newcommand{\bea}{\begin{eqnarray}}
\newcommand{\eea}{\end{eqnarray}}

\newcommand{\A}{\mathbb{A}}
\newcommand{\C}{\mathbb{C}}

\newcommand{\R}{\mathbb{R}}

\newcommand{\Z}{\mathbb{Z}}
\newcommand{\N}{\mathbb{N}}

\newcommand{\HH}{\mathbb{H}}

\newcommand{\PP}{\mathbb{P}}

\newcommand{\ra}{\rightarrow}

\newcommand{\dvol}{\operatorname{dv}}

\newcommand{\Imag}{\operatorname{Im}}

\newcommand{\Trace}{\operatorname{Trace}}

\newcommand{\Ad}{\textrm{Ad}}

\newcommand{\Lie}[1]{\mathfrak{#1}}

\newcommand{\tu}[1]{\textup{#1}}

\newcommand{\unitary}[1]{\textup{U$(#1)$}}

\def\co{\colon\thinspace}

\newcounter{mtheorem}

\begin{document}
\title{Calorons and constituent monopoles}
\author{Lorenzo Foscolo \and Calum Ross}
\address{Department of Mathematics, University College London, London WC1E 6BT, United Kingdom}
\thanks{{\tt l.foscolo@ucl.ac.uk}, {\tt calum.ross@ucl.ac.uk}}

\begin{abstract}
We study anti-self-dual Yang--Mills instantons on $\R^{3}\times S^{1}$, also known as calorons, and their behaviour under collapse of the circle factor. In this limit, we make explicit the decomposition of calorons in terms of constituent pieces which are essentially charge $1$ monopoles. We give a gluing construction of calorons in terms of the constituents and use it to compute the dimension of the moduli space. The construction works uniformly for structure group an arbitrary compact semi-simple Lie group.
\end{abstract}
\maketitle

\section{Introduction}

This paper is motivated by the study of the behaviour of 4-dimensional anti-self-dual Yang-Mills instantons under codimension-1 collapse. We focus on intantons on the flat model $\R^3 \times S^1$ where the circle has radius $\epsilon\ra 0$. In the literature these periodic instantons are often referred to as \emph{calorons}. We construct families of calorons that can be qualitatively described as superpositions of building blocks localised around points in the collapsed limit $\R^3$, glued into a singular $S^1$--invariant abelian background obtained from a sum of Dirac monopoles. All the calorons we produce have ``maximal symmetry breaking'' at infinity, \ie the centraliser of the holonomy around circles $\{x \}\times S^1$ for $|x|\gg 1$ is a maximal torus in the structure group $G$. The approximation in terms of simpler building blocks is increasingly accurate as $\epsilon\ra 0$ and we expect our construction captures some generic behaviour of instantons under codimension-1 collapse.

The building blocks in our construction are simple explicit ``fundamental'' calorons obtained from the charge $1$ $SU(2)$ monopole on $\R^3$ and suitable embeddings of $SU(2)$ into a higher rank compact semi-simple structure group $G$. For $G=SU(2)$ there are two different types of fundamental calorons: one is the charge $1$ monopole lifted to $\R^3\times S^1$ as a circle invariant instanton; the other type of fundamental caloron, that we call a ``rotated'' monopole, is \emph{not} circle invariant (in a way compatible with a fixed framing at infinity) and arises from the non-trivial loop in the moduli space of charge $1$ monopoles. For higher rank $G$, the fundamental calorons are obtained by embedding the charge $1$ monopole along one of the simple coroots and the rotated monopole along the lowest negative coroot.

We refer to Theorem \ref{thm:Existence} in the paper for a precise statement of our existence result. As a consequence, we establish the existence of calorons with non-trivial holonomy for arbitrary compact semisimple structure group.

\begin{theorem*} Fix a compact simply connected semi-simple Lie group $G$ of rank $\tu{rk}$ with Lie algebra $\Lie{g}$, a generic holonomy parameter $\omega\in \Lie{g}$, an instanton number $n_0\in \Z$ and a total magnetic charge $\gamma_\tu{m}=\sum_{\mu=0}^{\tu{rk}}{n_\mu\, \alpha_\mu^\vee}$ in the coroot lattice of $G$. If $n_\mu\geq 0$ for all $\mu=0,\dots, \tu{rk}$, then the moduli space of calorons $\mathcal{M}(\omega,\gamma_{\tu{m}},n_0)$ with structure group $G$ is non-empty.	
\end{theorem*}

Our construction of calorons is reminiscent of the description of ``widely separated'' monopoles on $\R^3$ \cite{Jaffe:Taubes,Taubes:G:Monopoles}. The interpretation of calorons in terms of constituent monopoles is not new, but a direct description in terms of the connection and for arbitrary structure group has not appeared before in the literature. In the late 1990s, implementing explicitly the Nahm Transform for calorons, Kraan--van Baal \cite{Kraan_1998} and Lee--Lu \cite{Lee_1998} independently produced an explicit family of $SU(2)$ calorons with non-trivial holonomy, instanton number 1 and vanishing total magnetic charge. These calorons are qualitatively interpreted as a superposition of a monopole and an anti-monopole. Conjectural decriptions of calorons with higher rank structure group in terms of constituent monopoles were then discussed in \cite{Kraan:1998sn,Kraan:1998gh} for $G=SU(n)$ and \cite{Lee2_1998} for general $G$. This idea and its relation with the collapsing behaviour of instantons does not appear to have been explored further and the purpose of this paper is to provide a simple but rigorous gluing construction implementing it.

More generally, besides early references such as \cite{HS:1978} that constructs explicit calorons with trivial holonomy, much of the work on calorons makes use of the Nahm Transform for $G=SU(n)$ calorons \cite{CH_2007,Takayama_2016} rather than working with the connection $\A$ directly. Some explicit solutions have been obtained using the Nahm Transform to construct multicalorons in \cite{Bruckmann:2002}, and symmetric configurations in \cite{Ward:2003sx,Harland_2007,Cork:2017,Kato:2020}. For an overview of the literature on calorons, including examples with trivial holonomy see \cite{Cork:2017}. Given the Nahm Transform only applies to classical structure groups, we use some of the tools and ideas of our gluing construction to also answer some basic open questions about the moduli space of calorons for arbitrary structure group. In particular, we calculate the expected dimension of the moduli space in Theorem \ref{thm:Dimension}.

\begin{theorem*} In the notation of the previous theorem,
\[
\tu{dim}\, \mathcal{M}(\omega,\gamma_\tu{m},n_0) = 4(n_0 + \dots + n_{\tu{rk}}).
\]	
\end{theorem*}

This uses the index theorem for Dirac operators on ALF manifolds \cite{Nye:Singer,Cherkis_2021,IndexFibredBoundary} and an excision argument based on our gluing construction. (The index theorem does not immediately apply to the deformation theory of calorons because the adjoint action of the holonomy at infinity is always trivial on the Cartan subalgebra.) This index computation shows that the calorons we construct depend on the right number of parameters (positions and phases of the constituent monopoles) and that the fundamental calorons are precisely the ones that belong to a 4-dimensional moduli space. Given moduli spaces of calorons are hyperk\"ahler manifolds, this is the smallest non-trivial number of parameters gauge equivalence classes of calorons can depend on.

Calorons are the simplest examples of instantons on ALF spaces, recently studied by Cherkis--Larra\'in-Hubach--Stern \cite{Cherkis_2021,Cherkis:Stern:II}. We expect similar results to the ones described here to hold in this more general setting. Moreover, we hope that the behaviour described here can be used to model codimension-4 curvature concentration of generalised instantons on sequences of higher dimensional manifolds with special holonomy undergoing codimension-1 collapse.

Our gluing construction could also be used to provide a description of an asymptotic region of the moduli space of calorons and of its asymptotic hyperk\"ahler geometry. While the metric is in general incomplete due to instanton bubbling, moduli spaces of calorons are expected to provide interesting examples of non-compact hyperk\"ahler spaces. For example, for $G$ simply-laced these spaces also arise as moduli spaces of vacua in quantum field theory (more precisely, Coulomb branches of certain 3-dimensional supersymmetric quiver gauge theories \cite{Braverman:2016,Nakajima:Takayama,Lee:2000hp}).

\subsubsection*{Plan of the paper} In Section~\ref{sec:bc_and_top_invariants} we fix the notations and conventions that we use and explain the asymptotics and topological invariants of a caloron. Section~\ref{sec:fundamental calorons} gives the definition of fundamental calorons. We show how to construct approximate calorons by gluing together Dirac monopoles and fundamental calorons in Section~\ref{sec:approximate solutions}. Section~\ref{sec:Analysis} provides all of the linear analysis results (in weighted H\"older spaces) that we need to study and deform our approximate calorons to an exact solutions. The proof of the main existence theorem is completed in Section~\ref{sec:existence} and the dimension formula is given in Section~\ref{sec:Index Computations}.

\subsubsection*{Acknowledgements} This work was supported by the Engineering and Physical Sciences Research Council [grant number EP/V047698/1]. The work of LF is also supported by a Royal Society University Research Fellowship. CR would like to thank Josh Cork for several useful discussions about $SU(n)$ calorons and the role of the rotation map. LF thanks Michael Singer for discussions in early stages of this work.

\subsubsection*{Declarations}
The authors have no competing interests to declare that are relevant to the content of this article.

\section{Boundary conditions and topological invariants}
\label{sec:bc_and_top_invariants}

In this brief preliminary section we fix the notation and conventions that will be used throughout the paper.

\subsubsection*{The base manifold} Fix coordinates $(x,t)$ on $\R^3\times \R$ and identify $\R^3\times S^1$ with $\R^3\times \R/2\pi\Z$. Fix $\epsilon>0$ and endow $\R^3\times S^1$ with the flat metric $g_\epsilon = g_{\R^3} + \epsilon^2 dt^2$ and volume form $\dvol_{g_\epsilon} = \epsilon\, dt\wedge\dvol_{\R^3}$. A \emph{caloron} is a connection $\A$ on a principal bundle over $\R^3\times S^1$ with anti-self-dual curvature with respect to $(g_\epsilon, \dvol_{g_\epsilon})$. In this paper we study calorons in the limit $\epsilon\ra 0$.

\subsubsection*{The structure group} Let $G$ be a compact semi-simple Lie group. Any principal $G$--bundle $P\ra \R^3\times S^1$ is trivial and therefore without loss of generality we assume that $G$ is simply connected.

We now collect some of the Lie theoretic notions we will need. Denote by $\Lie{g}$ the Lie algebra of $G$ and let $\langle \, \cdot\, ,\,\cdot\,\rangle_{\Lie{g}}$ be the Killing form of $\Lie{g}$ normalised so that long coroots have norm $\sqrt{2}$ (with the convention that all coroots are long if $\Lie{g}$ is simply laced).

Fix a maximal torus $T$ in $G$ with Lie algebra $\Lie{h}$, a Cartan subalgebra of $\Lie{g}$. Since $G$ is simply connected $T=\Lie{h}/\Lambda$, where $\Lambda$ is the coroot lattice of $\Lie{g}$. We also fix a choice of simple roots $\alpha_1,\dots,\alpha_{\rm rk}$ and corresponding coroots $\alpha_1^\vee,\dots, \alpha_{\rm rk}^\vee$. Here ${\rm rk}$ is the rank of $G$. We introduce the lowest root $\alpha_0$ and the corresponding coroot $\alpha_0^\vee$. We have
\begin{equation}\label{eq:Lowest:Coroot}
\alpha_0^\vee = -\sum_{\mu=1}^{\rm rk}{m_\mu\,\alpha_\mu^\vee}
\end{equation} 
for integers $m_1,\dots, m_{\rm rk}\in \Z_{>0}$ (sometime referred to as the dual Coxeter labels of $\Lie{g}$).

Given these data, we let $A^+$ denote the fundamental alcove, the simplex in $\Lie{h}$ defined by the inequalities
\begin{equation}\label{eq:Alcove}
\alpha_{\mu}(\xi) \geq 0,\quad  \mu=1,\dots, {\rm rk}\,\qquad \alpha_0(\xi)\geq -1.
\end{equation}
The fundamental alcove is the fundamental domain for the action of $W\ltimes \Lambda$ on $\Lie{h}$, where $W$ is the Weyl group of $\Lie{g}$. Note that the coroots $\alpha_0^\vee,\dots, \alpha_{\rm rk}^\vee$ are inward-pointing normals to the facets of the boundary of $A^+$, such as shown in Figure~\ref{fig:Fundamental Alcove} for the case of $G=SU(3)$. 

Finally, recall that the extended Cartan matrix $\widetilde{C}$ is defined by $\widetilde{C}_{\mu\nu}=\alpha_\nu (\alpha_\mu^\vee)$, $\nu,\mu=0,1,\dots, {\rm rk}$. It satisfies $\widetilde{C}_{\mu \mu}=2$ and $\widetilde{C}_{\mu\nu}\leq 0$ for $\mu\neq \nu$.

\subsubsection*{Boundary conditions}

In the following we will consider connections $\A$ on the trivial principal $G$--bundle $P=P_G\ra \R^3\times S^1$ asymptotic to the $S^1$--invariant abelian calorons we now define.
 
The complement of a compact set in $\R^3\times S^1$ retracts to $S^2\times S^1$, so principal $T$--bundles on such an exterior domain are in one-to-one correspondence with elements $\gamma_\tu{m}\in \Lambda$, \ie any such bundle $H^{\gamma_\tu{m}}$ must be the pull-back from $S^2$ of the $T$--bundle associated with the Hopf circle bundle $S^3\ra S^2$ and the group homomorphism $\exp {\gamma_\tu{m}}\co S^1\ra T$. We will refer to $\gamma_\tu{m}$ as the \emph{total magnetic charge}.

The bundle $H^{\gamma_\tu{m}}$ carries a distinguished connection $A^{\gamma_\tu{m}}$ with curvature $dA^{\gamma_\tu{m}} = \tfrac{1}{2}\gamma_{\tu{m}}\, \dvol_{S^2}$. Given the additional choice of $\omega\in \Lie{h}$ we consider the $S^1$--invariant instanton on $\left(\R^{3}\setminus\{0\}\right)\times S^{1}$ 
\begin{equation}\label{eq:Model:Infinity}
\A_\infty(\omega,\gamma_\tu{m}) = A^{\gamma_\tu{m}} + \epsilon\left( \epsilon^{-1}\omega + \Phi^{\gamma_\tu{m}}\right) dt,\qquad \Phi^{\gamma_\tu{m}} = -\tfrac{1}{2|x|}\gamma_\tu{m}.
\end{equation} 
The fact that $\A_\infty(\omega,\gamma_\tu{m})$ is an instanton on $(\R^3\setminus\{ 0\})\times S^1$ follows from the fact that $(A^{\gamma_\tu{m}},\Phi^{\gamma_\tu{m}})$ satisfies the Bogomolny equation $dA^{\gamma_m} = \ast_{\R^3} d\Phi^{\gamma_\tu{m}}$ on $\R^3\setminus \{ 0\}$.

Note that the parameter $\omega$ can be shifted by an arbitrary element $\xi\in\Lambda$ by a gauge transformation of the form $(x,t)\mapsto \exp (2\pi t\,\xi)$. Furthermore, if we regard $\A_\infty(\omega,\gamma_\tu{m})$ as a connection on the $G$--bundle $H^{\gamma_\tu{m}}\times_T G$, the action of constant gauge transformations in the normaliser $N(T)$ of $T$ in $G$ generate the action of the Weyl group $W$ on the pair $(\omega,\gamma_\tu{m})\in \Lie{h}\times\Lie{h}$. Using these degrees of freedom, we can therefore always move $\omega$ to lie in the fundamental alcove. In this paper we make the standing assumption that $\omega$ lies in the interior $\mathring{A}^+$ of the fundamental alcove. In particular, the limiting holonomy of $\A_\infty(\omega,\gamma_\tu{m})$ on circles $\{ x\}\times S^1$ for $|x|\ra\infty$ commutes only with elements in $T\subset G$, \ie we have \emph{maximal symmetry breaking} at infinity. We will refer to $\omega$ as the \emph{holonomy parameter}.

\begin{example}\label{eg:SU(n):Lie:theory}
	The reader might find it useful to keep in mind the explicit case where $G=SU(n)$. Then
	\[
	\omega = \tu{diag}(i\mu_1,\dots, i\mu_{n}), \qquad \gamma_\tu{m} = \tu{diag}(ik_1,\dots,ik_{n}),
	\] 
	with $\mu_i\in \R$ and $k_i\in \Z$ satisfying $\mu_1+\dots+\mu_{n}=0=k_1+\dots+k_{n}$. The condition $\omega\in \mathring{A}^+$ is
	\[
	\mu_1 > \mu_2 > \dots > \mu_{n} > \mu_1 -1.
	\] 
	In particular, for $n=2$, \ie $G=SU(2)$, the holonomy parameter is a single number $\mu_1=-\mu_2\in (0,\tfrac{1}{2})$ and the total magnetic charge is a single integer $k_1=-k_2$.
\end{example}   

\begin{figure}[htbp]
\begin{center}
\begin{tikzpicture}[thick,scale=1.5]
\draw (0,0) -- (4,0) node[pos=.5, below] {$\{\alpha_{1}=0\}$};
\draw (0,0) -- (2,3.46) node[pos=.5, sloped, above]  {$\{\alpha_{2}=0\}$};
\draw  (2,3.46) --(4,0) node[pos=.5, sloped, above]  {$\{\alpha_{0}=-1\}$};
\draw[->] (2,0) -- (2,1) node[pos=.5, right] {$\alpha_{1}^{\vee}$};
\draw[->] (1,1.73) -- (1.87,1.23) node[pos=.5, sloped, below] {$\alpha_{2}^{\vee}$};
\draw[->] (3,1.73) -- (2.13,1.23) node[pos=.5, sloped, below] {$\alpha_{0}^{\vee}$};
\end{tikzpicture}
\caption{The fundamental alcove $A^{+}$ for $SU(3)$.}
\label{fig:Fundamental Alcove}
\end{center}
\end{figure}
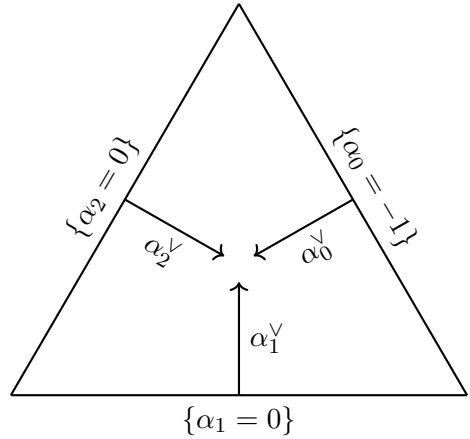

\subsubsection*{Instanton number}

Fix $\epsilon>0$ and $(\omega,\gamma_\tu{m})\in \mathring{A}^+\times\Lambda$ and consider a pair $(\A,f)$ consisting of a connection $\A$ on the trivial principal $G$--bundle $P\ra \R^3\times S^1$ and a \emph{framing} $f$ that identifies $P$ and $H^{\gamma_\tu{m}}\times_T G$ on $(\R^3\setminus B_R)\times S^1$ for some $R\gg 1$ and such that
\begin{subequations}\label{eq:Framing}
\begin{equation}
f^\ast \A = \A_\infty(\omega,\gamma_\tu{m}) + a,
\end{equation}
where 
\begin{equation}
|\nabla_{\A_\infty}^k a|=O(r^{-1-k+\nu})
\end{equation}
for some $\nu<0$ and all $k\geq 0$.
\end{subequations}

\begin{remark*}
By \cite[Theorem B]{Cherkis_2021} the much weaker asymptotic conditions of finite Yang--Mills energy and maximal symmetry breaking at infinity along a single ray in $\R^3$ force any caloron to satisfy the asymptotic conditions \eqref{eq:Framing} with $\nu=-2$. 	
\end{remark*}

To any such pair $(\A,f)$ we associate a topological number $n_0\in \N_0$ in either of the following equivalent ways, \cf \cite[\S 2]{Nye:Singer} and \cite[Chapter 2]{Nye_thesis}. Firstly, we can represent $(P,\A)$ as a pair $(\widetilde{P},\widetilde{\A})$ on $\R^3\times \R$ invariant under the action of $\Z$ generated by translations on $\R$ and an an isomorphism $\widetilde{P}|_{\{t=2\pi\}}\ra \widetilde{P}|_{\{t=0\}}$, \ie a smooth map $h\co \R^3\ra G$. We can choose this trivialisation in a compatible way with the framing $f$. Then $h$ is the identity outside a compact set and it extends to a map $h\co S^3=\R^3\cup\{ \infty\}\ra G$, whose degree we denote by $n_0$. Here by degree we mean the pull back of the generator of $H^3(G;\Z)$ in $H^3(S^3;\Z)\simeq \Z$. Alternatively, composing $f$ with a fixed trivialisation of $H^{\gamma_\tu{m}}\times_T G$ on $(\R^3\setminus B_R)\times S^1$ allows one to construct a new connection $\A'$ that is trivial outside a compact set. Then the closed form $\langle F_{\A'}\wedge F_{\A'}\rangle_\Lie{g}$ defines a compactly supported cohomology class, which we identify with an integer $n_0$ by integration. The integer $n_0$ so defined will be called the \emph{instanton number} of $(\A,f)$.

We consider the set $\mathcal{A}_\epsilon(\omega,\gamma_\tu{m},n_0)$ of pairs $(\A,f)$ satisfying the boundary conditions \eqref{eq:Framing} and with instanton number $n_0$. The space of framed connections is acted upon by the group $\mathcal{G}$ of gauge transformations that are asymptotic to the identity at infinity (with suitable polynomial decay). The quotient 
\[
\mathcal{M}_\epsilon (\omega,\gamma_\tu{m},n_0) = \{ (\A,f)\in \mathcal{A}_\epsilon(\omega,\gamma_\tu{m},n_0)\, |\, \ast F_\A  = -F_\A\}/\mathcal{G} 
\]     
is the moduli space of (framed) calorons.

Using \eqref{eq:Lowest:Coroot}, define integers $n_1,\dots, n_{\rm rk}$ by
\begin{equation}\label{eq:Monopole:Charges} 
\gamma_\tu{m} = \sum_{\mu=0}^{\rm rk}{n_\mu\,\alpha_\mu^\vee} = (n_1-n_0\, m_1)\,\alpha_1^\vee + \dots + (n_{\rm rk}-n_0\, m_{\rm rk})\,\alpha_{\rm rk}^\vee.
\end{equation}
The purpose of this paper is to interpret the integers $(n_0,n_1,\dots, n_{\rm rk})$ as the number of ``constituent monopoles'' of a caloron in $\mathcal{M}_\epsilon$.

\begin{remark*}
The Yang--Mills energy of a caloron is given by (\cf \cite[\S 2.1.7]{Nye_thesis} for $G$ a unitary group)
\begin{equation}
\mathcal{YM}(\A)=\tfrac{1}{8\pi^{2}}\Vert F_{\A}\Vert^{2}_{L^{2}}=n_{0}\left(1+ \alpha_{0}\left(\omega\right)\right)+\sum_{\mu=1}^{\rm rk}\tfrac{1}{2}\Vert\alpha_{\mu}^{\vee}\Vert_{\mathfrak{g}}^{2}\, n_{\mu}\,\alpha_{\mu}\left(\omega\right).\label{eq:caloron energy expression}
\end{equation}
\end{remark*}

\section{The Fundamental Calorons}
\label{sec:fundamental calorons}

In this section we introduce the simple model solutions that will be used as building blocks in the construction of more complicated calorons. These ``fundamental'' calorons are all obtained from the simplest non-abelian solution of the Bogomolny equation on $\R^3$, the charge 1 BPS (Bogomolny--Prasad--Sommerfield) $SU(2)$ monopole. The fundamental calorons correspond to BPS monopoles embedded along the simple coroots of the structure group $G$, and a ``rotated'' BPS monopole embedded along the lowest negative root. Here a ``rotated'' BPS monopole is a caloron obtained by acting on the BPS monopole by a $t$--dependent large gauge transformation which generates the \emph{rotation map} of \cite[\S 2.2]{Nye_thesis}. In other words, the ``rotated'' BPS monopole is the caloron corresponding to a non-trivial loop in the moduli space of (framed) charge 1 monopoles.

\subsection{Fundamental \texorpdfstring{$SU(2)$}{SU(2)} calorons}
The simplest case to consider is that of $SU(2)$ calorons, where there are just two types of fundamental calorons. As in Example \ref{eg:SU(n):Lie:theory}, $SU(2)$ calorons are classified by their magnetic charge $k$ and instanton number $n_0$, as well as the holonomy parameter $\omega\in \left(0,\frac{1}{2}\right)$. The two fundamental calorons are the $(k,n_0)=(1,0)$ BPS monopole and the $(k,n_0)=(-1,1)$ ``rotated'' BPS monopole. We begin the section with these fundamental calorons and then describe how to obtain fundamental calorons for higher rank Lie group via embeddings of $\Lie{su}_2$ in $\Lie{g}$. 

\subsubsection{The charge 1 BPS monopole}
For $SU(2)$ the charge $1$ BPS monopole with mass $v>0$ is the explicit solution $(A_{\rm BPS},\Phi_{\rm BPS})$ of the Bogomolny equation $F_{A}=\ast_{\R^3}d_{A}\Phi$ given by
\begin{equation}
\Phi_{\rm BPS}=\left(v\coth\left(2vr\right)-\frac{1}{2r}\right)\hat{x}\cdot i\vec{\tau}\qquad A_{{\rm BPS},i}=-\frac{1}{2}\left(1-\frac{2vr}{\sinh\left(2vr\right)}\right)\varepsilon_{ija}\frac{\hat{x}^{j}i\tau^{a}}{r},\label{eq:BPS monopole formula}
\end{equation}
with $r=|x|$ the radial distance from the origin in $\R^3$, $\hat{x}=r^{-1}x$ and $\vec{\tau}$ the vector of Pauli matrices.

If $\omega\in (0,\tfrac{1}{2})$ we set $v=\epsilon^{-1}\omega$ in \eqref{eq:BPS monopole formula} and obtain an $S^1$--invariant caloron $\A_{\rm BPS}^+$ on $\R^3\times S^1$ by 
\begin{equation}
\A_{\text{ BPS}}^{+}=A_{\text{BPS}}+\epsilon\, \Phi_{\text{BPS}} dt.\label{eq:Bps monopole as a caloron}
\end{equation}


This connection is put in an asymptotically abelian gauge by the bundle map $f_{\rm BPS}^+\co H\times_{S^1}SU(2)\ra P=S^{3}/S^{1} \times SU(2)$, where $H$ is the Hopf circle bundle and $f_{\rm BPS}^{+}[\left(p,g\right)]=\left([p],pg\right)$ for $p\in S^{3}$ and $g\in SU(2)$.
Following \cite[\S IV.7]{Jaffe:Taubes} and working in the local coordinates of the standard trivialisation of $H$ over the north hemisphere, $f_{\rm BPS}^+$ is given by 
\[
f_{\rm BPS}^+=\cos(\tfrac{1}{2}\theta)\,\tu{id}-i\sin (\tfrac{1}{2}\theta)\,\vec{\varpi}\cdot\vec{\tau}\qquad \text{with } \vec{\varpi}=\begin{pmatrix}
-\sin\phi\\
\cos\phi\\
0
\end{pmatrix}
\]
An analogous formula holds in the standard trivialisation of $H$ over the southern hemisphere.
\begin{remark*}
In terms of the associated vector bundles,
\begin{equation}
f^{+}_{\rm BPS}:L\oplus L^{-1}\to E\vert_{\R^{3}\backslash \{0\}}=\left(\R^{3}\backslash \{0\}\right)\times\C^{2}
\end{equation}
relates the BPS monopole on the trivial bundle $\left(\R^{3}\backslash \{0\}\right)\times\C^{2} $ away from the origin to the Dirac monopole on $S^{2}$ (with a Dirac string going through the south pole in the local trivialisation above). Here $L=\mathcal{O}(1)$ is the standard complex line bundle on $S^{2}$ and $L$ and $L^{-1}$ are the eigenbundles of the asymptotic Higgs field.
\end{remark*}

The role of $f^{+}_{\rm BPS}$ is made precise in the following proposition.
\begin{prop}\label{prop:BPS asymptotics}
Given $\omega\in (0,\tfrac{1}{2})$ there exists $r_\infty(\epsilon)\propto \epsilon$ such that outside of $B_{r_\infty(\epsilon)}(0)\times S^1$  the pair $(\A^+_{\rm BPS},f_{\text{BPS}}^{+})$ satisfies
\[
\left(f^{+}_{{\rm BPS}}\right)^{*}\A_{\rm BPS}^{+}=\A_\infty (\omega,1)+a_{{\rm BPS}}^+,
\]
with $r^k \vert \nabla_{\A_\infty}^k a_{\text{BPS}}^+\vert \leq C \epsilon^{-1} e^{-c\,\epsilon^{-1} r}$ for all $k\geq 0$ and $\epsilon$--independent constants $C,c>0$.
\proof
This is just the statement that a non-abelian monopole has the asymptotics of a Dirac monopole. For example,
\[
\left( f_{\rm BPS}^+\right)^{-1}\Phi_{\text{BPS}}\, f_{\rm BPS}^+=\left(v-\frac{1}{2r} + O\left(vr^{-4vr}\right) \right)i\tau^{3}.
\]
\endproof
\end{prop}

In other words, the gauge equivalence class of the pair $(\A^+_{\rm BPS},f_{\text{BPS}}^{+})$ lies in the moduli space $\mathcal{M}^{SU(2)}_\epsilon(\omega, 1, 0)$ of $SU(2)$ calorons with holonomy parameter $\omega$, magnetic charge $1$ and vanishing instanton number. Here the instanton number vanishes since $\A^+_{\rm BPS}$ is $S^1$--invariant.

\begin{remark*}
	By pulling $\A^+_{\rm BPS}$ back by a translation in $\R^3$ and composing $f^+_{\rm BPS}$ with an automorphism of $(H,\A_\infty)$, \ie a constant phase, we obtain an exhaustive $4$-parameter family of inequivalent framed calorons in $\mathcal{M}^{SU(2)}_\epsilon(\omega, 1, 0)$.
\end{remark*}

\subsubsection{The rotation map}

The other fundamental $SU(2)$ caloron is the ``rotated'' BPS monopole.

Fix $\omega\in (0,\tfrac{1}{2})$ and consider the BPS monopole $(A_{\rm BPS},\Phi_{\rm BPS})$ \eqref{eq:BPS monopole formula} with mass $v = \epsilon^{-1}\left( \tfrac{1}{2}-\omega\right)>0$. Let $\hat{\Phi}\co \R^3\ra\Lie{su}_2$ be a smooth map satisfying
\[
\hat{\Phi}(x) = \frac{\Phi_{\rm BPS}(x)}{\vert \Phi_{\rm BPS}(x)\vert}
\]
outside a compact set. Consider the $t$--dependent family $g_t(x)=g(x,t)$ of ``large'' gauge transformations (\ie gauge transformations on $\R^3$ that do not converge to the identify at infinity)
\begin{equation}
g(x,t)=\exp\left(-\tfrac{1}{2} t\,\hat{\Phi}(x)\right):\R^{3}\times \R\to SU(2).\label{eq:rotation map large gauge transformation}
\end{equation}
For a map $q\co \R^3 \ra SU(2)$ let $P_{q}$ be the principal $SU(2)$--bundle on $\R^3\times S^1$ defined by $P_q=(\R^3\times \R\times SU(2))/\Z$, where the action of $\Z$ is generated by $(x,t,g)\mapsto (x, t+2\pi, q(x)g)$. We then regard $g$ in \eqref{eq:rotation map large gauge transformation}
 as a bundle morphism $g\co P_h\ra P_{-1}$, where $h(x):=-g(x,2\pi)^{-1}$. Since outside a compact set $\hat{\Phi}$ takes value in the adjoint orbit of $i\tau_3$, note that $h(x)=1$ in a neighbourhood of infinity. As shown in \cite[\S 2.2]{Nye_thesis}, the extension of $h$ as a map $h\co S^3\ra SU(2)$ has degree $1$.

Now, since the adjoint action of $-1$ is trivial, we regard $\A^+_\tu{BPS}=A_{\rm BPS} + \epsilon\, \Phi_{\rm BPS}\, dt$ as a connection on $P_{-1}$ and then define a caloron $\A^{-}_{\rm BPS}$ on $P_h$ by  
\begin{subequations}
\begin{equation}\label{eq:Rotated BPS monopole}
\A^{-}_{\rm BPS}=g^\ast \A^{+}_{\text{BPS}}.
\end{equation}
We can also define a framing for $\A^{-}_{\rm BPS}$ from the framing $f^+_{\rm BPS}$ for $\A^{+}_{\text{BPS}}$. The only subtlety is that we need to introduce the action of a constant gauge transformation such as $i\tau_2$, that acts on the Cartan subalgebra of $\Lie{su}_2$ as the non-trivial element of the Weyl group, to ensure that the holonomy parameter lies in the fundamental alcove. 
More precisely, let $H^{-1}$ be the inverse of the Hopf line bundle $H$, and let $H_{-1}$ denote the $S^{1}$ bundle defined analogously to $P_{-1}$, \ie it is the radial extension of the principal $S^1$--bundle on $S^2 \times S^1$ defined by $(S^3\times \R)/\Z$ with $\Z$--action generated by $(p,t)\mapsto (-p, t+2\pi)$.
Then introduce the bundle map
\[
g_\infty\co \co H^{-1}\times_{S^1}SU(2)\ra H_{-1}\times_{S^1} SU(2), \qquad g_\infty = \exp\left( -\tfrac{1}{2}it\tau_3\right)i\tau_2.
\]  
On the exterior domain where $\hat{\Phi}=|\Phi_{\rm BPS}|^{-1}\Phi_{\rm BPS}$ we then define the framing
\begin{equation}\label{eq:Framing:Rotated BPS monopole}
f_{\rm BPS}^- \co H^{-1}\times_{S^1}SU(2)\ra P_h, \qquad f^-_{\rm BPS} = g^{-1}\circ f^+_{\rm BPS} \circ g_\infty,	
\end{equation}
where we regard $f^+_{\rm BPS}$ as a bundle map $f_{\rm BPS}^+\co H_{-1}\times_{S^1}SU(2)\ra P_{-1}$.
\end{subequations}

The following proposition follows immediately from Proposition \ref{prop:BPS asymptotics} and summarises the main properties of $(\A^-_{\rm BPS},f^-_{\rm BPS})$.

\begin{prop}\label{prop:rotated BPS asymptotics}
Given $\omega\in (0,\tfrac{1}{2})$ there exists $r_\infty(\epsilon)\propto \epsilon$ such that outside of $B_{r_\infty(\epsilon)}(0)\times S^1$  the pair $(\A^-_{\rm BPS},f_{\text{BPS}}^{-})$ satisfies
\[
\left(f^{-}_{{\rm BPS}}\right)^{*}\A_{\rm BPS}^{-}=\A_\infty (\omega,-1)+a_{{\rm BPS}}^-,
\]
with $r^k \vert \nabla_{\A_\infty}^k a_{\text{BPS}}^-\vert \leq C \epsilon^{-1} e^{-c\,\epsilon^{-1} r}$ for all $k\geq 0$ and $\epsilon$--independent constants $C,c>0$.
\end{prop}

In other words, the gauge equivalence class of the pair $(\A^-_{\rm BPS},f_{\text{BPS}}^{-})$ lies in the moduli space $\mathcal{M}^{SU(2)}_\epsilon(\omega, -1, 1)$ of $SU(2)$ calorons with holonomy parameter $\omega$, magnetic charge $-1$ and instanton number $1$ (since we already observed that the clutching map $h$ has degree 1).

\begin{remark*}
Also in this case translations in $\R^3$ and composition of the framing with an automorphism of $(H^{-1},\A_\infty)$	yield an exhaustive 4-parameter family of calorons in $\mathcal{M}^{SU(2)}_\epsilon(\omega, -1, 1)$. Since $\A^-_{\rm BPS}$ is not $S^1$--invariant, the circle action on $\R^3\times S^1$ lifts to a circle action on $\mathcal{M}^{SU(2)}_\epsilon(\omega, -1, 1)$ which corresponds to changing the framing. 
\end{remark*}

\begin{remark*}
Since it has negative magnetic charge, $\A^-_{\rm BPS}$ is referred to as an anti-monopole in \cite{Kraan_1998, Lee_1998}. We find that referring to it as a ``rotated'' monopole is less misleading. 
\end{remark*}

\subsection{Higher rank groups}

For a simple Lie group of rank ${\rm rk}>1$ the fundamental calorons given above generalise and we have a BPS monopole for every simple root and a rotated BPS monopole for the lowest negative root.
These fundamental calorons are found by embedding the fundamental $SU(2)$ calorons into $G$ as $T'\times SU(2)$ calorons, for $T'$ a torus of rank ${\rm rk}-1$.

Let $\Lie{h}'$ denote the Lie algebra of $T'$. Recall that every positive root $\alpha$ of $\Lie{g}$ corresponds to a Lie algebra embedding $\rho\co \Lie{h}'\oplus\Lie{su}_2\ra \Lie{g}$ with $\rho(0,i\tau_3)=\alpha^\vee$ and $\rho (\Lie{h}'\oplus \{ 0\})=\ker\alpha$. By abuse of notation we identify $\rho$ with the induced group homomorphism $T'\times SU(2)\ra G$. We let $\rho_1,\dots, \rho_{\rm rk}$ denote the homomorphisms corresponding to the simple roots $\alpha_1,\dots,\alpha_{\rm rk}$ and let $\rho_0$ be the one corresponding to the highest root $-\alpha_0$.    

Now, fix $\omega\in \mathring{A}^+$. For each $\mu=0,1,\dots, {\rm rk}$ we decompose $\omega = \omega'_\mu + \tfrac{1}{2}\alpha_\mu (\omega)\,\alpha_\mu^\vee$ in the decomposition $\Lie{h}=\ker\alpha_\mu \oplus \R\, \alpha_\mu^\vee$. Note that the assumption $\omega\in \mathring{A}^+$ implies that $-\tfrac{1}{2}\alpha_0(\omega)$ and $\tfrac{1}{2}\alpha_\mu(\omega)$ for $\mu=1,\dots, {\rm rk}$ are real numbers lying in $(0,\frac{1}{2})$.

For $\mu=1,\dots, {\rm rk}$ we now consider the $SU(2)$ caloron $\A^+_{\rm BPS}$ with holonomy parameter $\tfrac{1}{2}\alpha_\mu (\omega)$ and then set
\begin{subequations}\label{eq:Fundamental Calorons}
\begin{equation}\label{eq:Fundamental Calorons Simple roots}
\A_\mu(\omega) = \rho_\mu \left( \A^+_{\rm BPS} + \omega'_\mu \, dt\right).	
\end{equation}
Similarly, we set
\begin{equation}\label{eq:Fundamental Calorons Lowest root}
\A_0(\omega) = \rho_0 \left( \A^-_{\rm BPS} + \omega'_0 \, dt\right) 	
\end{equation}
for $\A^-_{\rm BPS}$ the $SU(2)$ caloron with holonomy parameter $-\tfrac{1}{2}\alpha_0(\omega)$.
\end{subequations}
For each such caloron we also have a framing $f_\mu$ induced by $f^\pm_{\rm BPS}$.
\begin{remark*}For $G=SU(n)$ there is a large gauge transformation which gives an isomorphism between the moduli spaces 
\begin{equation}
\mathcal{M}^{SU(n)}_{\epsilon}\left(\bar{\omega},\alpha_{1}^{\vee},0\right)\leftrightarrow \mathcal{M}^{SU(n)}_{\epsilon}\left(\omega,\alpha_{0}^{\vee},1\right),
\end{equation}
where $\bar{\omega}\in \mathring{A}^{+}$ is a holonomy parameter related to $\omega$ through the large gauge transformation. For $G=SU(2)$ we saw above that $\bar{\omega}=\frac{1}{2}-\omega$. This large gauge transformation is also called the rotation map in the literature, \eg in \cite{Nye_thesis,Cork_thesis}. At the level of the extended Dynkin diagram this isomorphism is explicitly a rotation cycling the simple roots of the extended Dynkin diagram. Under the above isomorphism $\A_{0}\left(\omega\right)$ is the image of $\A_{1}\left(\bar{\omega}\right)$.
 \end{remark*}

\begin{remark*}
In an abelian gauge on $\left(\R^{3}\backslash\{0\}\right)\times S^{1}$ we can write
\begin{equation}
\A_{\mu}(\omega)=A_{\mu}+\epsilon\,\Phi_{\mu}dt=A_{\mu}+\epsilon\left(\epsilon^{-1}\omega_{\mu}' +\varphi\,\alpha_{\mu}^{\vee}\right)dt,
\end{equation}
where $\varphi = |\Phi_\tu{BPS}|$ for a BPS monopole of the appropriate mass. Since $\epsilon\varphi (x)\ra \frac{1}{2}\alpha_{\mu}\left(\omega\right)$ as $|x|\ra \infty$ and $\varphi (0)=0$, the Higgs field gives a map into the Cartan subalgebra $\epsilon\Phi_{\mu}:\R^{3}\setminus\{0\}\to \Lie{h}$ which parametrises a straight line from $\omega\in \mathring{A}^{+}$ (for large $x$) to the component of the boundary of the alcove $A^{+}$ with normal $\alpha_\mu^\vee$ (in the limit $x=0$), \cf 
Figure~\ref{fig:SU(3)_caloron} for the case $G=SU(3)$. Another way to say this is that as $\vert x\vert\to \infty$ the gauge group breaks to the maximal torus $T$, while near the origin there is a symmetry enhancement to $\rho_{\mu}\left(SU(2)\times T'\right)$. 
\end{remark*}

\begin{figure}[htbp]
\begin{center}
\begin{tikzpicture}[thick,scale=1.5]
\draw (0,0) -- (4,0) node[pos=.5, below] {$\{\alpha_{1}=0\}$};
\draw (0,0) -- (2,3.46) node[pos=.5, sloped, above]  {$\{\alpha_{2}=0\}$};
\draw  (2,3.46) --(4,0) node[pos=.5, sloped, above]  {$\{\alpha_{0}=-1\}$};
\draw[dashed] (1.5,0) -- (1.5,0.8) ;
\draw[dashed] (0.72,1.25) --  (1.5,0.8) ;
\draw[dashed] (3.03,1.68) --  (1.5,0.8) ;
\draw[fill] (1.5,0.8)  circle [radius =0.05cm] node[above]{$\omega$};
\end{tikzpicture}
\caption{Images of the Higgs fields $\epsilon \Phi_\mu\co \R^3\setminus\{ 0\}\ra \Lie{h}$ of the fundamental $SU(3)$ calorons.}
\label{fig:SU(3)_caloron}
\end{center}
\end{figure}

\begin{remark*}
The reason for these particular choices of $(\A^\pm_{\rm BPS},\rho_\mu)$, \ie why we do not take different combinations and more general embeddings of $\Lie{h}'\oplus\Lie{su}_2$ in $\Lie{g}$, is inspired by \cite{Lee2_1998}. It appears unmotivated at the moment, but we will see in Section~\ref{sec:Index Computations}	 that these are the only choices yielding 4-dimensional moduli spaces, hence justifying referring to the $\A_\mu$'s as ``fundamental'' calorons.
\end{remark*}

The embedded BPS caloron \eqref{eq:Fundamental Calorons} for a root $\alpha$ has the asymptotics of an abelian $S^1$--invariant caloron where the Higgs field is the one of a Dirac monopole along the coroot $\alpha^{\vee}$ with a singularity at the origin,
\[
\Phi= \epsilon^{-1}\omega - \frac{1}{2r}\alpha^{\vee}.
\]
More formally, the following proposition is an immediate consequence of Propositions \ref{prop:BPS asymptotics} and \ref{prop:rotated BPS asymptotics}.
\begin{prop}\label{prop:G caloron asymptotics}
Given $\omega\in (0,\tfrac{1}{2})$ there exists $r_\infty(\epsilon)\propto \epsilon$ such that outside of $B_{r_\infty(\epsilon)}(0)\times S^1$  the pair $(\A_\mu (\omega),f_\mu)$ of \eqref{eq:Fundamental Calorons} satisfies
\[
\left(f_{\mu}\right)^{*}\A_{\mu}(\omega)=\A_\infty (\omega,\alpha_\mu^\vee)+a_{\rm BPS,\mu},
\]
with $r^k \vert \nabla_{\A_\infty}^k a_{\rm BPS,\mu} \vert \leq C \epsilon^{-1} e^{-c\,\epsilon^{-1} r}$ for all $k\geq 0$ and $\epsilon$--independent constants $C,c>0$. Moreover, $(\A_\mu(\omega),f_\mu)$ has instanton number $n_0=1$ if $\mu=0$ and $n_0=0$ otherwise.
\end{prop}

In the rest of the paper we will refer to $\A_\mu(\omega)$ as the fundamental caloron of ``type'' $\alpha_\mu^\vee$ with holonomy parameter $\omega$.

\begin{remark}\label{rmk:G caloron asymptotics}
The framing $f_\mu$ with the properties of Proposition \ref{prop:G caloron asymptotics} is uniquely defined up to an element of $\tu{Aut}(H^{\alpha_\mu^\vee},\A_\infty)\simeq T$. Note however that framings related by an element of the subgroup $\rho_\mu (T'\times\{ 1\})\simeq \tu{Aut}(P_{SU(2)}\times_{\rho_\mu}G,\A_\mu(\omega))$ yield gauge equivalent framed calorons. Hence $f_\mu$ is uniquely defined up to an element $\psi\in\rho_\mu (\{ 1\}\times U(1))\simeq \tu{Aut}(H^{\alpha_\mu^\vee},\A_\infty)/ \tu{Aut}(P_{SU(2)}\times_{\rho_\mu}G,\A_\mu(\omega))$, where $U(1)$ is the maximal torus of $SU(2)$.	
\end{remark}

\begin{remark*}
For uniformity of notation, if $G=SU(2)$ we set $\A_1(\omega)=\A^+_{\rm BPS}$ and $\A_0(\omega) = \A^-_{\rm BPS}$.	
\end{remark*}

\section{Approximate Solutions}
\label{sec:approximate solutions}

The idea of our result is to build a caloron by gluing the fundamental solutions of the previous section into a singular background configuration. In this section we describe this singular background and then use fundamental calorons to produce a smooth connection that satisfies the self-duality equations only in an approximate sense. In the next two sections we will then use analysis to deform this approximate solution into an actual caloron.

\subsection{The initial singular abelian solution}

The singular background solution is an $S^1$--invariant abelian caloron obtained from a sum of Dirac monopoles on $\R^3$.

Recall that given a point $p\in \R^3$ and a charge $\gamma\in \Lambda$ we have a Dirac monopole $(A^\gamma_p,\Phi^\gamma_p)$ on $\R^3\setminus \{ p\}$ with
\[
\Phi^\gamma_p = -\tfrac{1}{2|x-p|}\gamma, \qquad dA^\gamma_p = \ast_{\R^3} d\Phi^\gamma_p
\]
We then obtain a caloron $A^\gamma_p + \epsilon\, \Phi^\gamma_p\, dt$ on the bundle $H^\gamma_p\ra (\R^3\setminus\{ p\})\times S^1$. Since principal torus bundles form a group, given distinct points $p,p'\in\R^3$ and charges $\gamma,\gamma'\in\Lambda$ we can also ``add'' the two Dirac monopoles to obtain a caloron $A^\gamma_p + A^{\gamma'}_{p'} + \epsilon\, \left( \Phi^\gamma_p + \Phi^{\gamma'}_{p'}\right) dt$ on the bundle $H^\gamma_p \times_{(\R^3\setminus \{ p,p'\})\times S^1}H^{\gamma'}_{p'}$ on $(\R^3\setminus \{ p,p'\})\times S^1$.

Now, fix $\omega\in\mathring{A}^+$ and non-negative integers $n_0,n_1,\dots, n_{\rm rk}\geq 0$. We then define a total magnetic charge $\gamma_\tu{m}$ as in \eqref{eq:Monopole:Charges}. Consider $n:=\sum _{\mu=0}^{{\rm rk}}n_{\mu}$ distinct points $p^{1}_{0},\dots, p^{n_0}_{0},\dots, p^{1}_{\rm rk},\dots, p^{n_{\rm rk}}_{\rm rk}\in \R^{3}$. We fix $d_{\rm min},d_{\rm max}>0$ such that $\vert p^{i}_{\mu}-p^{j}_{\nu}\vert>d_{\rm min}$ for each distinct pair of points and all the points are contained in $B_{d_{\rm max}}(0)\subset \R^3$. In the rest of the paper all constants are allowed to depend on $\omega,d_{\rm min},d_{\rm max}$ without further notice and will be uniform in the positions of the $n$ points provided the bounds given by $d_{\rm min}$ and $d_{\rm max}$ remain satisfied.

We now define the $S^1$--invariant abelian caloron 
\begin{equation}\label{eq:Singular:Caloron}
\A_{\text{sing}}=A_{\text{sing}}+\epsilon\, \Phi_{\text{sing}}\,dt = \sum_{\mu=0}^{\rm rk}\sum_{i=1}^{n_{\mu}}A^{\alpha_\mu^\vee}_{p^i_\mu} + \epsilon \left( \epsilon^{-1}\omega + \sum_{\mu=0}^{\rm rk}\sum_{i=1}^{n_{\mu}}\Phi^{\alpha_\mu^\vee}_{p^i_\mu}\right) dt
\end{equation}
on a bundle $P_\tu{sing}$ over $\left(\R^{3}\backslash\{p_{1}^{0},\dots,p_{n_{r}}^{r}\}\right)\times S^{1}$ with structure group the maximal torus $T$ of $G$. Of course, we can also regard $\A_\tu{sing}$ as a connection on the $G$--bundle $P_\tu{sing}\times_{T}G$.

\begin{example}\label{eg:Singular:SU(2)}
When $G=SU(2)$, $\A_{\rm sing}$ is simply the superposition of the flat connection $i\omega\,\tau_3$, for $\omega\in (0,\tfrac{1}{2})$, with $n_1$ Dirac monopoles of charge $1$ and $n_0$ Dirac monopoles of charge $-1$.
\end{example}

We will now collect some of the properties of $\A_\tu{sing}$. First of all, consider the behaviour of $\A_\tu{sing}$ at infinity. It follows immediately from \eqref{eq:Singular:Caloron} and the explicit formula for the Dirac monopole that the holonomy parameter and total magnetic charge of $\A_{\rm sing}$ are precisely $\omega$ and $\gamma_\tu{m}$ respectively. The next proposition describes instead the singular behaviour of $\A_{\rm sing}$ near $p^i_\mu$. Set $r^i_\mu:=|x-p^i_\mu|$. 

\begin{prop}\label{prop:Singular:Caloron:singularity}
There exists $r_0>0$ and a bundle isomorphism $f^i_\mu\co H^{\alpha_i^\vee}\ra P_{\rm sing}$ over $B_{r_0}(p^i_\mu)\times S^1$ such that
\[
(f^i_\mu)^\ast\A_{\rm sing} = \A_\infty (\omega^i_\mu,\alpha_\mu^\vee) + a^i_\mu
\]
with $(r^i_\mu)^k \vert \nabla_{\A_\infty}^k a^i_\mu\vert \leq C r^i_\mu$ for all $k\geq 0$ and $\omega^i_\mu = \omega + O(\epsilon)$.
\proof
Write
\[
A_{\rm sing} = A^{\alpha_\mu^\vee}_{p^i_\mu} + \sum_{(\nu,j)\neq (\mu,i)}A^{\alpha_\nu^\vee}_{p^j_\nu}, \qquad \Phi_{\rm sing} = \epsilon^{-1}\omega + \Phi^{\alpha_\mu^\vee}_{p^i_\mu} + \sum_{(\nu,j)\neq (\mu,i)}\Phi^{\alpha_\nu^\vee}_{p^j_\nu}.
\]
A classical multipole expansion centred at a point away from the singularity allows one to estimate the term $\sum_{(\nu,j)\neq (\mu,i)}\Phi^{\alpha_\nu^\vee}_{p^j_\nu}$. The holonomy parameter $\omega^i_\mu$ is defined using the constant term in this expansion:
\[
\epsilon^{-1}\omega^i_\mu = \epsilon^{-1}\omega - \sum_{(\nu,j)\neq (\mu,i)}\frac{1}{2|p^j_\nu-p^i_\mu|}\alpha_{\nu}^{\vee}.
\]  
Solving the Bogomolny equation in a radial gauge centred at $p^i_\mu$ then defines the bundle map $f^i_\mu$ and allows one to estimate $\sum_{(\nu,j)\neq (\mu,i)}A^{\alpha_\nu^\vee}_{p^j_\nu}$ in terms of the control of the Higgs field.
\endproof
\end{prop}

A final simple but crucial observation is that $\A_{\rm sing}$ is abelian in the following uniform quantitative sense away from the singularities.

\begin{lemma}\label{lem:Singular:Caloron:Holonomy:Estimate}
There exists $\epsilon_0, \sigma>0$ such that for $\epsilon\in (0,\epsilon_0)$ there exists $r_0(\epsilon)\propto \epsilon$ such that outside of $\bigcup_{\mu,i}{B_{r_0(\epsilon)}(p^i_\mu)}\times S^1$ we have
\[
\alpha _\mu (\epsilon\,\Phi_{\rm sing}) \geq \sigma>0 \mbox{ for all }\mu=1,\dots,{\rm rk}, \qquad \alpha_0 (\epsilon\,\Phi_{\rm sing})\geq -1+\sigma>-1.
\]

\end{lemma}
In other words, away from the singularities $\epsilon\,\Phi_{\rm sing}$ takes values in a fixed compact subset of $\mathring{A}^+$.

\begin{proof}
Observe that $\alpha(\epsilon\, \Phi_{\rm sing})$ is a harmonic function on $\R^{3}\backslash\{p_{0}^{1},\dots,p^{n_{\tu{rk}}}_{\tu{rk}}\}$ for any $\alpha\in\Lie{h}^\ast$.
Hence, by the maximum/minimum principle on the complement of $\bigcup_{\mu,i}{B_{r_0(\epsilon)}(p^i_\mu)}$, it suffices to check that the inequalities are satisfied as $|x|\ra \infty$ and on the interior boundaries $\{ r^i_\mu=r(\epsilon)\}$. Now, for large $|x|$ we have
\[
\alpha(\epsilon\, \Phi_{\rm sing}) \approx \alpha (\omega),
\]
while Proposition \ref{prop:Singular:Caloron:singularity} implies that near $p^i_\mu$ we have 
\[
\alpha(\epsilon\, \Phi_{\rm sing}) \approx \alpha (\omega^i_\mu) - \frac{\epsilon}{2r^i_\mu} \alpha (\alpha_\mu^\vee).
\]
Since $\omega\in\mathring{A}^+$ we can choose $\sigma$ so that the inequalities in the statement of the lemma are satisfied near infinity. Since $\omega^i_\mu = \omega + O(\epsilon)$, for $\epsilon$ sufficiently small we can also assume that the same inequalities are satisfied by $\omega^i_\mu$ instead of $\omega$. Finally, in order to take care of the singular term at $p^i_\mu$, fix $c>0$ sufficiently large so that $\frac{1}{2c}|\alpha (\alpha_\mu^\vee)|\leq \tfrac{1}{2}\sigma$ for $\alpha=\alpha_{\nu}$ for any $\nu=0,\dots,{\rm rk}$. Then, up to decreasing $\epsilon,\sigma$ slightly if neecessary, we can assume that the inequalities in the statement of the lemma are also satisfied for $r^i_\mu = c\,\epsilon=:r_0(\epsilon)$.
\end{proof}

\subsection{Desingularisation}
As the caloron $\A_{\text{sing}}$ is manifestly singular at the $p^{i}_{\mu}$, to find an approximate non-singular caloron on all of $\R^{3}\times S^{1}$ we need to glue in non-abelian calorons that match the singular behaviour of $\A_{\text{sing}}$ asymptotically. These are the fundamental calorons of the previous section.

For $R=R(\epsilon)\in (0,d_{\rm min})$ to be fixed later, decompose $\R^{3}\times S^{1}$ as
\[
\R^{3}\times S^{1}=U_{{\rm sing}}\cup \bigsqcup_{\mu,i}U^i_{\mu},
\]
where
\[
U_{{\rm sing}}=\left(\R^{3}\backslash \bigsqcup_{\mu,i}B_{\frac{R}{2}}(p^{i}_{\mu})\right)\times S^{1},\qquad 
U^i_{\mu}=B_{R}(p^{i}_{\mu})\times S^{1}.
\]
Up to the $S^1$--factor, these open sets intersect in a disjoint union of annuli centred at the $p^i_{\mu}$'s.

By restriction, we think of $(P_{\rm sing}\times_T G,\A_{\rm sing})$ as a bundle with connection on $U_{\rm sing}$. Similarly, for each $\mu=0,\dots, {\rm rk}$ and $i=1,\dots,n_\mu$ we identify $B_R(p^i_\mu)$ with $B_R(0)\subset \R^3$ and endow $U^i_\mu$ with the pair $(P_{SU(2)}\times_{\rho_\mu}G,\A_\mu(\omega^i_\mu))$ where $\A_\mu(\omega^i_\mu)$ is the fundamental caloron of type $\alpha_\mu^\vee$ and holonomy parameter $\omega^i_\mu$ defined in Proposition \ref{prop:Singular:Caloron:singularity}.

By Propositions \ref{prop:G caloron asymptotics} and \ref{prop:Singular:Caloron:singularity} on the overlap $U_{\rm sing}\cap U^i_\mu$ there are isomorphisms $f^i_\mu, f_\mu$ of $P_{\rm sing}\times_T G$ and $P_{SU(2)}\times_{\rho_\mu}G$ with $H^{\alpha_i^\vee}\times_T G$ such that
\[
(f^i_\mu)^\ast\A_{\rm sing} = \A_\infty (\omega^i_\mu,\alpha_\mu^\vee) + a^i_\mu, \qquad f_\mu^\ast\A_\mu(\omega^i_\mu) = \A_\infty (\omega^i_\mu,\alpha_\mu^\vee) + a_{\rm BPS,\mu}.
\]
We have the additional freedom to choose a gluing parameter $\psi^i_\mu\in \rho_\mu (\{ 1\}\times U(1))$. This gluing parameter is there to line up the framings of the $\A_{\text{sing}}$ and $\A_{\mu}$. It is $U(1)$ valued rather than $T$ valued due to Remark \ref{rmk:G caloron asymptotics}. We can then define a smooth $G$--bundle $P$ on $\R^3\times S^1$ identifying $P_{\rm sing}\times_T G\ra U_{\rm sing}$ and $P_{SU(2)}\times_{\rho_\mu}G\ra U^i_\mu$ over the overlap $U_{\rm sing}\cap U^i_\mu$ via
\begin{subequations}\label{eq:Approximate:Solution}
\begin{equation}\label{eq:Approximate:Solution:Clutching}
f_\mu\circ \psi^i_\mu \circ (f^i_\mu)^{-1}\co P_{\rm sing}\times_T G\ra P_{SU(2)}\times_{\rho_\mu}G.
\end{equation}
On $P$ we define a connection $\A'_\epsilon = \A'_\epsilon (\omega,\{ (p^i_\mu,\psi^i_\mu)\}_{\mu,i})$ as follows. Fix a smooth bump function $\chi$ to interpolate between $\chi(r)=1$ for $r\in[0,\frac{R}{2})$ and and $\chi(r)=0$ for $r\geq R$ and set $\chi^i_\mu(x,t) = \chi(r^i_\mu)$. We set
\begin{equation}\label{eq:Approximate:Solution:Connection}
\A_{\epsilon}'=\begin{dcases}
\mathbb{A}_{\mu}(\omega^i_\mu) & \mbox{ if } r^i_\mu \leq \tfrac{1}{2}R,\\
\mathbb{A}_\infty (\omega^i_\mu,\alpha_\mu^\vee)+\chi^i_{\mu}\,(\psi^i_\mu)^\ast a_{\text{BPS},\mu}+\left(1-\chi^i_{\mu}\right)a^{i}_{\mu} &  \mbox{ if } \tfrac{1}{2}R\leq r^i_\mu \leq R,\\
\mathbb{A}_{\text{sing}} & \mbox{ otherwise}.
\end{dcases}
\end{equation}
\end{subequations}

\subsubsection{Estimates of the error}
The connection $\A'_\epsilon$ is an approximate caloron rather than a true caloron since it does not satisfy $F_{\A'_{\varepsilon}}^{+}=0$ on the overlaps $U_{\rm sing}\cap U^i_\mu$.

Recall from Propositions \ref{prop:G caloron asymptotics} and \ref{prop:Singular:Caloron:singularity} that $a^{i}_{\mu}=O\left(r_{i}^{\mu}\right)$ and $(\psi^i_\mu)^\ast a_{\text{BPS},\mu}=O\left(\epsilon^{-1}e^{-c\,\epsilon^{-1}r^i_\mu}\right)$. We define $R(\epsilon)$ implicitly by
\[
R(\epsilon) = \epsilon^{-1}e^{-c\,\epsilon^{-1}R(\epsilon)}
\]
 so that these two contributions to the error have comparable size. Note that
\begin{equation}\label{eq:Gluing:Region}
R(\epsilon)\approx \epsilon\,\vert\ln\epsilon\vert.
\end{equation}
In particular, as $\epsilon\ra 0$ we have $R(\epsilon)\ra 0$ and $\epsilon^{-1}R(\epsilon)\ra \infty$. Thus as $\epsilon\ra 0$ the sets $U^\epsilon_{\rm sing}$ form an exhaustion of $\left(\R^{3}\backslash\{p_{1}^{0},\dots,p_{n_{r}}^{r}\}\right)\times S^{1}$. On the other hand, if we rescale $U^i_\mu$ by $\epsilon^{-1}$ we obtain an exhaustion of $\R^3\times S^1$.

\begin{lemma}\label{lemma:bounds on Fplus} Let $\A'_\epsilon$ be the approximate caloron defined in \eqref{eq:Approximate:Solution}. Then the self-dual part of the curvature satisfies
\[
\vert F_{\A'_\epsilon}^{+}\vert \leq \frac{1}{r^{i}_\mu}\max\left(\vert a_{\rm BPS,\mu}\vert, \vert a^{i}_{\mu}\vert\right)+\max\left(\vert a_{\rm BPS,\mu}\vert^{2}, \vert a^{i}_{\mu}\vert^{2}\right)=O(1) \qquad \text{on }U_{\rm sing}\cap U^i_\mu.
\]
\end{lemma}
\begin{proof}
This follows form a direct computation. We have
\[
\A'_\epsilon=\mathbb{A}_\infty (\omega^i_\mu,\alpha_\mu^\vee)+a, \qquad a=\chi^{\mu}_{i}(\psi^i_\mu)^\ast a_{\text{BPS}}+\left(1-\chi^{\mu}_{i}\right)a^{i}_{\mu,\text{sing}}
\]
with $\mathbb{A}_\infty (\omega^i_\mu,\alpha_\mu^\vee)$ a caloron. Hence $F_{\A'_\epsilon}^+ = d_{\A_\infty}^+ a + \tfrac{1}{2}[a,a]^+$. In order to estimate $d_{\A_\infty}^+ a$ we use the fact that
\[
d_{\A_\infty}^+ \widetilde{a}= -\tfrac{1}{2}[\widetilde{a},\widetilde{a}]^+
\]
for $\widetilde{a}=a_{\rm BPS,\mu}$ and $\widetilde{a}=a^i_\mu$ (since $\A_\infty + \widetilde{a}$ is a caloron in either case), together with the fact that $|\nabla\chi^i_{\mu}|=O\left((r^i_{\mu})^{-1}\right)$.
\end{proof}

Thus the error $F_{\A'_\epsilon}^+$ is uniformly bounded in $\epsilon$, but it is also supported on a region of increasingly small size in the same limit and in this sense we can say that the error is increasingly small as $\epsilon\ra 0$. For example, $\| F_{\A'_\epsilon}^+\|^2_{L^2}\leq C\epsilon\, R(\epsilon)^3 =O(\epsilon^4 |\ln \epsilon|^3)$. On the other hand, as $\epsilon\ra 0$ the metric $g_\epsilon$ and connection $\A'_{\epsilon}$ degenerate, so it is not immediately clear that $\A'_\epsilon$ can be deformed to a genuine caloron for small $\epsilon>0$. In  the next section, we will introduce weighted H\"older spaces which are better suited to do the analysis in this degenerate limit and then use a quantitative version of the Implicit Function Theorem to deform $\A'_\epsilon$ to a nearby caloron $\A_\epsilon$.

\section{Linear Analysis}\label{sec:Analysis}

In this section we collect some fundamental results about mapping properties of the linear operators appearing in the deformation theory of calorons and in particular study dependence of constants on $\epsilon$ when we couple these operators to the approximate caloron $\A'_\epsilon$ of the previous section. In order to obtain uniform estimates, all the analyis is carried out in appropriate weighted H\"older spaces.

\subsubsection*{The operators} The deformation complex of an instanton $\A$ on a $G$--bundle $P\ra M^4$ is
\begin{subequations}
\begin{equation}\label{eq:Deformation:Complex}
0\ra \Omega^0(M;\tu{ad}\, P)\stackrel{d_\A}{\longrightarrow} \Omega^1(M;\tu{ad}\, P) \stackrel{d_\A^+}{\longrightarrow} \Omega^+(M;\tu{ad}\, P)\ra 0.
\end{equation}
In our set-up $M$ is (the complement of finitely many curves $\{ p\}\times S^1$ in) $\R^3\times S^1$. From \eqref{eq:Deformation:Complex} we deduce that the first-order operator governing the deformation theory of an instanton is
\begin{equation}\label{eq:Dirac:Deformation:Complex}
D_\A=d_\A^+\oplus d_\A^\ast\co \Omega^1(M;\tu{ad}\, P)\ra\Omega^+(M;\tu{ad}\, P)\oplus\Omega^0(M;\tu{ad}\, P).\end{equation}
\end{subequations}
We will study its mapping properties via the second-order operator $D_\A D_\A^\ast$. The Weitzenb\"ock formula (see \eg \cite{Freed:Uhlenbeck}) reads
\begin{equation}\label{eq:Weitzenbock}
D_\A D_\A^\ast = \nabla_{\A}^\ast\nabla_\A + F_\A^+,
\end{equation} 
where the action of $F_\A^+$ is a zeroth-order operator obtained from the Lie bracket in $\Lie{g}$ and, via identifications $\Lambda^0\oplus \Lambda^+\simeq \HH$ and $\Lambda^+\simeq\Imag\HH$, quaternionic multiplication. If $M$ is hyperk\"ahler (as in our set up) then the bundle $\Lambda^+T^\ast M$ is trivialised by parallel section and for an instanton $F_\A^+=0$. Therefore we will start the section with a discussion of the operator $\nabla_\A^\ast\nabla_\A$ acting on sections of the adjoint bundle.

\subsection{Fredholm theory}

In this subsection we establish results about the Bochner Laplacian $\nabla_\A^\ast\nabla_\A$ in the simplest situation where $\A$ is a fixed smooth connection on $\R^3\times S^1$ and the metric $g_\epsilon$ on $\R^3\times S^1$ is assumed fixed. In the second part of this section we will adapt these results to $\A = \A'_\epsilon$ and discuss dependence of constants on $\epsilon$.

Assume therefore that $\A$ is a connection on a (trivial) $G$--bundle $P$ over $\R^3\times S^1$ satisfying the boundary conditions \eqref{eq:Framing} for some $(\omega,\gamma_\tu{m})\in \mathring{A}^+\times\Lambda$.

\subsubsection{Weighted H\"older spaces}

By abuse of notation, let $r$ denote a smooth $S^1$--invariant function on $\R^3\times S^1$ with $r\geq 1$ and $r\approx |x|$ on $(\R^3\setminus B_{d})\times S^1$ for some $d> 1$. (When we apply the results of this section to $\A'_\epsilon$ we will require $d\geq d_{\rm max}$ so that all the singularities of $\A_{\rm sing}$ are contained in $B_d \times S^1$.) For example, we can take $r = \sqrt{1+|x|^2}$.

\begin{definition}\label{def:Caloron:Holder}
Given $k\in \Z_{\geq 0}$, $\alpha\in (0,1)$ and $\nu\in \R$, we define the $C^{k,\alpha}_\nu$--norm of a section $u$ of the (trivial) adjoint bundle $\tu{ad}\, P$ on $\R^3\times S^1$ by
\[
\sum_{j=0}^k{\| r^{-\nu+j} \nabla^j_{\A}u \|_{L^\infty}} + \sup_{\tu{dist}(p,p') \leq r(p,p')}{ r(p,p')^{-\nu +k}}\frac{|\nabla_{\A}^k u (p) - \nabla_{\A}^k u (p')|}{\tu{dist}(p,p')^\alpha},
\]
where norms are defined using the metric $g_\epsilon$, $r(p,p')=\min\{ r(p),r(p')\}$ and the difference of $\tu{ad}\, P$--valued tensors $\nabla_{\A}^k u (p) - \nabla_{\A}^k u (p')$ is computed using the parallel transport of the connection induced by $\A$ and the Levi--Civita connection of $g_\epsilon$. The same definition with the last term dropped defines the $C^{k}_\nu$--norm of $u$. The Banach spaces $C^{k,\alpha}_\nu$ and $C^{k}_{\nu}$ are defined as the closure of $C^{\infty}_c$ with respect to the corresponding norm. 
\end{definition}

Immediate consequences of the definition and the fact that $\R^3 \times S^1$ has cubic volume growth are the continuous embedding
\begin{equation}\label{eq:L2:embedding}
C^{0,\alpha}_\nu\subset L^2 \Longleftrightarrow \nu <-\tfrac{3}{2}
\end{equation}
and the integration-by-parts formula
\begin{equation}\label{eq:Integration:parts}
\langle \nabla_\A^\ast\nabla_\A u, v\rangle _{L^2}= \langle \nabla_\A u, \nabla_\A v\rangle_{L^2} \qquad \forall u\in C^{2,\alpha}_{\nu_1+1}, v\in C^{1,\alpha}_{\nu_2+1} \mbox{ with } \nu_1+\nu_2 < -3.
\end{equation}
Later in the paper, in order to control non-linearities in the equations we will also make use of the fact that any bounded pointwise bilinear form defines a continuous map
\begin{equation}\label{eq:Product:Weighted:Holder}
C^{0,\alpha}_{\delta_1}\times C^{0,\alpha}_{\delta_2} \longrightarrow C^{0,\alpha}_{\delta_1+\delta_2}.
\end{equation}
Combined with the compactness of the embedding $C^{1,
\alpha}(\Omega)\subset C^{0,\alpha}(\Omega)$ for a bounded domain $\Omega$, one can further deduce that multiplication by an element $u_1\in C^{0,\alpha}_{\delta_1}$ defines a \emph{compact} operator
\begin{equation}\label{eq:Comopact:Embedding:Weighted:Holder}
u_1 \times\, \cdot\, \co C^{1,\alpha}_{\delta_2} \longrightarrow C^{0,\alpha}_{\delta_1+\delta_2}	.
\end{equation}

We want to study the mapping properties of the bounded operator $\nabla_\A^\ast\nabla_\A \co C^{2,\alpha}_{\nu+1}\ra C^{0,\alpha}_{\nu-1}$. We begin with the following weighted Schauder estimates.

\begin{prop}\label{prop:Caloron:Schauder}
Given $\delta\in \R$, there exists a constant $C$ such that
\[
\| u\|_{C^{2,\alpha}_{\delta}} \leq C \left( \| \nabla_\A^\ast \nabla_\A u \|_{C^{0,\alpha}_{\delta-2}} + \| u \|_{C^0_{\delta}}\right)
\]	
for all $u\in C^{2,\alpha}_{\delta}$.
\proof
It is enough to show that every point $p$ has neighbourhoods $U_p\subset U'_p$ such that
\[
\| u\|_{C^{2,\alpha}_{\delta}(U_p)} \leq C \left( \| \nabla_\A^\ast \nabla_\A u \|_{C^{0,\alpha}_{\delta-2}(U'_p)} + \| u \|_{C^0_{\delta}(U'_p)}\right)
\]	
for a $p$--independent constant $C$. If $p$ lies in a compact subset of $\R^3\times S^1$ then the local estimate is simply the local Schauder estimate for the elliptic operator $\nabla_\A^\ast \nabla_\A$. We can therefore assume that $p=(x,t)$ satisfies $|x|\gg 1$.

The weighted norms we have defined are well-behaved under scaling and therefore we will obtain the local estimate by rescaling to a fixed situation. There are however two slight complications to take into account: $\R^3\times S^1$ is not scale invariant because of the compact factor and similarly, because of the non-vanishing constant term in the expansion of $\Phi_{\gamma_m,v}$ at infinity, the connection $\A$ is also not ``scale-invariant'' (\ie it is not $0$-homogeneous in the sense of \cite[Appendix B]{FHN}). Both of these issues are resolved by passing to the universal cover $\R^3\times\R$ of $\R^3\times S^1$, as we now explain.

For $4R=|x|$, set $B=B_R(0)\subset \R^3$ and for any $\eta\in (0,\infty)$ define $\eta B = B_{\eta R}(0)$. We then consider domains $U_p = \left( 2B\setminus \tfrac{1}{2}B\right) \times [-R,R]$ and $U'_p = \left(3B\setminus \tfrac{1}{3}B\right)\times [-2R,2R]$ in the universal cover $\R^3\times \R$ of $\R^3\times S^1$. If $R$ is sufficiently large, we can also assume that $\Phi = \partial_t\lrcorner\A$ satisfies $|\Phi|\geq c>0$ outside of $\tfrac{1}{3}B\times S^1$. Working on the universal cover, we act on $\A$ by the gauge transformation $\exp{\left( \omega \frac{\Phi}{|\Phi|}t\right)}$ to obtain a new connection $\A'$. Taking into account that $\nabla_{A_{\gamma_m}}\Phi_{\gamma_m,v}=O(r^{-2})$ and $|t|/r$ is uniformly bounded on any set where $|t|\lesssim |x|$, we deduce that $\A'$ is uniformly bounded in $C^\infty_{-1}$ on $U'_p$.

We can now rescale by $R$: up to a factor of $R^{-\delta}$, all norms coincide with norms on the fixed subset $B_3 \setminus B_{\frac{1}{3}}\times [-2,2]\subset \R^3\times \R$ defined using the standard flat metric and a rescaled connection which is uniformly bounded in $C^\infty$. The local estimates around $p$ now follow from standard Schauder estimates for this rescaled problem.
\endproof
\end{prop}

\subsubsection{Mapping properties}

Since the model connection $\A_\infty$ has reduced structure group $T\subset G$ and the kernel of $[\omega,\,\cdot\,]$ reduces to the Cartan subalgebra $\Lie{h}$ because of our assumption $\omega\in \mathring{A}^+$, for $r\gg 1$ we can decompose any section $u$ of $\tu{ad}\, P$ into its ``diagonal'' and ``off-diagonal'' components: $u=u_0 + u_\perp$, where $u_0$ has value in the trivial bundle with fibre $\Lie{h}$ and $u_\perp$ has values in the sum of line bundles $\bigoplus_{\alpha\in R^+}{H^{\gamma_m}\times_{\alpha}\C}$. By Fourier decomposition in the circle variable, we can further decompose $u_0 = u_0^0 + u_0'$ into $S^1$--invariant and oscillatory parts and therefore write $u=u_0^0 + u_0' + u_\perp$. Since $\A_\infty$ is reducible and $S^1$--invariant, the operator $\nabla_\A^\ast \nabla_\A$ preserves asymptotically this decomposition.

Now, the crucial observations is that for all $x\in\R^3$ with $r$ sufficiently large we have pointwise ``Poincar\'{e}-type'' estimates on $S^1_x:=\{ x \} \times S^1$ of the form
\begin{equation}\label{eq:Poincare:OffDiagonal:Oscillatory}
    \sigma\| u\|_{C^0(S^1_x)} \leq \| \nabla^{\mathbb{A}}_{\partial_t}u\|_{C^0(S^1_x)}
\end{equation}
for all $u=u_0' + u_\perp$ and a uniform constant $\sigma>0$. The existence of such a constant can be easily deduced by a contradiction argument using the fact that there are no $\mathbb{A}$--parallel sections on $S^1_x$ other than constant ``diagonal'' sections. We give a more constructive argument to show the dependence of $\sigma$. Consider first the case $u=u'_0$. Since $u$ has mean value zero on $S^1_x$, it must vanish at some point in this circle. Assuming this point is $t=0$ by a rotation, the fundamental theorem of calculus implies
\[
|u(x,t)| \leq \int_0^{2\pi}{|\partial_t u(x,t)|\, dt} \leq 2\pi\| \partial_t u\|_{C^0(S^1_x)} = 2\pi\| \nabla^{\mathbb{A}_\infty}_{\partial_t} u\|_{C^0(S^1_x)}.
\]
The argument for $u_\perp$ is similar. We have an orthogonal decomposition $u_\perp=\sum_\alpha u_\perp^\alpha$. The connection $\mathbb{A}_\infty$ preserves this decomposition. Moreover, restricted to the $\alpha$--factor it defines a flat connection on the trivial complex line bundle on $S^1_x$ with holonomy parameter $\omega^\alpha_x = \alpha(\omega) - \frac{1}{2r}\alpha(\gamma_\tu{m})$. For $r>R(\omega,\gamma_\tu{m})$ we see that $\omega^\alpha_x$ is never an integer. Hence, using parallel transport for the connection $\mathbb{A}_\infty$ restricted to $S^1_x$, $e^{i\omega^\alpha_x t}u_\perp^\alpha (x,\, \,\cdot\,)$ has vanishing mean value on $[0,2\pi]$ and we can apply the same argument as in the case of $u'_0$.  We conclude that \eqref{eq:Poincare:OffDiagonal:Oscillatory} holds for $\mathbb{A}_\infty$. Since $\mathbb{A}$ is asymptotic to $\mathbb{A}_\infty$, up to increasing $R$ and changing constants slightly, the same estimates hold for $\mathbb{A}$.


\begin{prop}\label{prop:Caloron:Fredholm}
Given $\delta\in  \R\setminus \Z$, there exists a constant $C$ and a compact set $K\subset \R^3\times S^1$ such that 
\[
\| u\|_{C^{2,\alpha}_\delta} \leq C \left( \| \nabla_\A^\ast \nabla_\A u \|_{C^{0,\alpha}_{\delta-2}} + \| u \|_{C^0(K)}\right)
\]	
for all $u\in C^{2,\alpha}_{\delta}$. In particular, $\nabla_\A^\ast \nabla_\A\co C^{2,\alpha}_\delta\ra C^{0,\alpha}_{\delta-2}$ is Fredholm for every $\delta\in \R\setminus \Z$.
\proof
The fact that the estimate implies the Fredholm property is standard so we only provide a proof for the estimate. For any $R>0$ sufficiently large, denote by $\Omega_R$ the exterior domain $\Omega_R = \{ r > R\}\subset \R^3\times S^1$. We will prove the estimate with $K=\Omega_R^c$ for $R$ sufficiently large. On such an exterior domain we can work in the decomposition $u=u_0^0 + u_0' + u_\perp$. Since $\nabla_\A^\ast \nabla_\A$ preserves this decomposition up to terms with arbitrarily small operator norm, it suffices to prove the estimate separately for $u_0' + u_\perp$ and $u_0^0$.

Now, an immediate consequence of \eqref{eq:Poincare:OffDiagonal:Oscillatory} is that
\begin{equation}\label{eq:BPS:Poincare}
\| u \|_{C^0_\delta(\Omega_R)} \leq CR^{-1} \| \nabla_\A u\|_{C^0_{\delta-1}(\Omega_R)} 
\end{equation}
whenever $u_0^0\equiv 0$. Then taking $R$ even larger if necessary we can deduce the estimate of Proposition \ref{prop:Caloron:Fredholm} directly from Proposition \ref{prop:Caloron:Schauder}. Note we do not need to assume $\delta\notin \Z$ for this.   

On the other hand, the action of $\nabla_\A^\ast \nabla_\A$ on the $S^1$--invariant diagonal component $u_0^0$ is asymptotic to the Laplacian of $\R^3$. Standard theory of elliptic operators on asymptotically conical manifolds (see for example the summary in \cite[Appendix B]{FHN}) implies that the estimate of Proposition \ref{prop:Caloron:Fredholm} holds provided $\delta$ is not one of the indicial roots of the Laplacian on $\R^3$, which are known to be all the integers \cite{Folland}.
\endproof 	
\end{prop}

\begin{prop}\label{prop:Caloron:Isomorphism}
If $\delta\in (-1,0)$ then $\nabla_\A^\ast \nabla_\A\co C^{2,\alpha}_\delta\ra C^{0,\alpha}_{\delta-2}$ is an isomorphism.
\proof
If $\delta<-\tfrac{1}{2}$ then the integration by parts formula \eqref{eq:Integration:parts} shows that any element in the kernel of $\nabla_\A^\ast\nabla_A$ is parallel and hence vanishes since it decays. Thus $\nabla_\A^\ast \nabla_\A\co C^{2,\alpha}_\delta\ra C^{0,\alpha}_{\delta-2}$ is injective for $\delta<-\tfrac{1}{2}$.

In order to conclude the proof, we need to use two facts that are part of the standard Fredholm package in weighted spaces (see for example \cite[Appendix B]{FHN}):
\begin{enumerate}
\item the cokernel of $\nabla_\A^\ast \nabla_\A\co C^{2,\alpha}_\delta\ra C^{0,\alpha}_{\delta-2}$	is isomorphic to the kernel of $\nabla_\A^\ast \nabla_\A$ in $C^{\infty}_{-1-\delta}$;
\item the kernel and cokernel of $\nabla_\A^\ast \nabla_\A\co C^{2,\alpha}_\delta\ra C^{0,\alpha}_{\delta-2}$ are locally constant in $\delta\in \R\setminus\Z$.
\end{enumerate}
The first statement uses the integration by parts formula and the weighted elliptic regularity of Proposition \ref{prop:Caloron:Schauder}. The second statement uses the asymptotic behaviour of elements in the kernel of the model operator $\nabla_{\A_\infty}^\ast\nabla_{\A_\infty}$ implicit in the proof of Proposition \ref{prop:Caloron:Fredholm}. 

Claim (i) and the injectivity result at the beginning of the proof imply that $\nabla_\A^\ast \nabla_\A\co C^{2,\alpha}_\delta\ra C^{0,\alpha}_{\delta-2}$ is surjective for $\delta>-\tfrac{1}{2}$. Then claim (ii) concludes the proof. 
\endproof
\end{prop}

Arguments analogous to the ones appearing in the proof of Propositions \ref{prop:Caloron:Schauder} and \ref{prop:Caloron:Fredholm} yield corresponding results for the first order operator $D_\A$ of \eqref{eq:Dirac:Deformation:Complex}.

\begin{prop}\label{prop:Caloron:Fredholm:Dirac}
	$D_\A\co C^{1,\alpha}_{\delta-1}\ra C^{0,\alpha}_{\delta-2}$ is a Fredholm operator for all $\delta\in \R\setminus\Z$, it is surjective for $\delta>-1$ and for $\delta\in (-1,0)$ its kernel coincides with its $L^2$--kernel.
\end{prop}
The last statement is an immediate application of Proposition \ref{prop:Caloron:Isomorphism}, weighted elliptic regularity and the fact that the indicial roots are the integers.

\subsection{Uniform estimates} We now extend the previous analysis to the situation where the metric $g_\epsilon$ degenerates as $\epsilon\ra 0$. For $\omega\in \mathring{A}^+$ and any $\epsilon>0$ sufficiently small, consider the approximate caloron $\A'_\epsilon$ constructed in \eqref{eq:Approximate:Solution}. From Proposition \ref{prop:Caloron:Isomorphism} we know that $\nabla_{\A'_\epsilon}^\ast \nabla_{\A'_\epsilon} \co C^{2,\alpha}_\delta\ra C^{0,\alpha}_{\delta-2}$ is an isomorphism for all $\delta\in (-1,0)$ and now we want to establish uniform estimates for its inverse as $\epsilon\ra 0$. In order to achieve this, we need to define a family of $\epsilon$--dependent norms on $C^{k,\alpha}_\nu$, equivalent to the norm of Definition \ref{def:Caloron:Holder} for fixed $\epsilon>0$ but that take into account the fact that the ambient geometry and the connection $\A'_\epsilon$ degenerate as $\epsilon\ra 0$. Recall the constant $d_{\rm min}, d_{\rm max}$ giving bounds on the minimum and maximum distance between the singularities $p^i_\mu$ of $\A_{\rm sing}$.  

\begin{definition}\label{def:Approximate:Weight} Define a weight function $r_\epsilon$ interpolating smoothly between
\[
r_\epsilon = \begin{dcases}
 \sqrt{\epsilon^2 + (r_\mu^i)^2}	 & \mbox{ if } r_\mu^i \leq \tfrac{1}{2}d_{\rm min} \mbox{ for some }\mu,i,\\
 1 & \mbox{ if } r_\mu^i \geq d_{\rm min} \mbox{ for all }\mu,i \mbox{ and } r\leq d_{\rm max},\\
 r & \mbox{ if } r\geq 2d_{\rm max}.
 \end{dcases}
\]	
\end{definition}

\begin{definition}\label{def:Approximate:Holder}
Given $k\in \Z_{\geq 0}$, $\alpha\in (0,1)$ and $\delta\in \R$, we define the $C^{k,\alpha}_\delta$--norm of a section $u$ of the adjoint bundle $\tu{ad}\, P$ on $\R^3\times S^1$ by
\[
\sum_{j=0}^k{\| r_\epsilon^{-\delta+j} \nabla^j_{\A'_\epsilon}u \|_{L^\infty}} + \sup_{\tu{dist}_\epsilon(p,p') \leq r_\epsilon(p,p')}{ r_\epsilon(p,p')^{-\delta +k}}\frac{|\nabla_{\A'_\epsilon}^k u (p) - \nabla_{\A'_\epsilon}^k u (p')|}{\tu{dist}_\epsilon(p,p')^\alpha},
\]
where norms are defined using the metric $g_{\epsilon}$, $r_\epsilon(p,p')=\min\{ r_\epsilon(p),r_\epsilon(p')\}$ and the difference of tensors $\nabla_{\A'_\epsilon}^k u (p) - \nabla_{\A'_\epsilon}^k u (p')$ is computed using the parallel transport of the connection induced by $\A'_\epsilon$ and the Levi--Civita connection of $g_\epsilon$. The same definition with the last term dropped defines the $C^{k}_\delta$--norm of $u$. The Banach spaces $C^{k,\alpha}_\delta$ and $C^{k}_{\delta}$ are defined as the closure of $C^{\infty}_0$ with respect to the corresponding norm. 
\end{definition}

\begin{prop}\label{prop:Approximate:Schauder}
Given $\delta\in \R$, there exists a constant $C$ independent of $\epsilon$ such that
\[
\| u\|_{C^{2,\alpha}_\delta} \leq C \left( \|\nabla_{A'_\epsilon}^\ast\nabla_{\A'_\epsilon} u \|_{C^{0,\alpha}_{\delta-2}} + \| u \|_{C^0_\delta}\right)
\]	
for all $u\in C^{k+2,\alpha}_{\delta}$.
\proof
The proof is completely analogous to the proof of Proposition \ref{prop:Caloron:Schauder}. The independence of the constant $C$ from $\epsilon$ follows from the invariance of the norms of Definition \ref{def:Approximate:Holder} under rescalings and passing to covers. We deduce local weighted Schauder estimates near the gluing regions by observing that on regions $B_R (p^i_\mu)\times S^1$ the triple $(g_\epsilon, \A'_\epsilon, r_\epsilon)$ is equivalent after rescaling to an essentially fixed triple $(g_1,\A,r_1)$ on $B_{\epsilon^{-1}R}(0)\times S^1$, where $\A$ is a small deformation of the fundamental caloron $\A_\mu (\omega^i_\mu)$ of \eqref{eq:Fundamental Calorons}. Away from the gluing regions, we obtain uniform local weighted Schauder estimate by working on the $\epsilon^{-1}$--cover $\R^3 \times \R/2\pi\epsilon^{-1}\Z$ of $\R^3\times S^1$. 
\endproof
\end{prop}

\begin{remark}\label{rmk:Approximate:Schauder}
The same proof yields uniform weighted Schauder	estimates
\[
\| u\|_{C^{k+2,\alpha}_\delta} \leq C \left( \|\nabla_{A'_\epsilon}^\ast\nabla_{\A'_\epsilon} u \|_{C^{k,\alpha}_{\delta-2}} + \| u \|_{C^0_\delta}\right)
\]	
for all $k\geq 0$.
\end{remark}

\begin{prop}\label{prop:Linearised}
The operator $\nabla_{\A'_\epsilon}^\ast\nabla_{\A'_\epsilon}\co C^{2,\alpha}_\delta\ra C^{0,\alpha}_{\delta-2}$ is an isomorphism for all $\delta\in (-1,0)$ and there exist $\epsilon_0,C>0$ such that
\[
\| u \|_{C^{2,\alpha}_\delta} \leq C \| \nabla_{\A'_\epsilon}^\ast\nabla_{\A'_\epsilon} u\|_{C^{0,\alpha}_{\delta-2}}
\]	
for all $u\in C^{2,\alpha}_\delta$ and all $\epsilon\in (0,\epsilon_0)$.
\proof
The fact that the operator is an isomorphism is Proposition \ref{prop:Caloron:Isomorphism}, so the main task is to establish the estimate.

We argue by contradiction. Using Proposition \ref{prop:Approximate:Schauder} we therefore assume that there exists sequences $\epsilon_k\ra 0$ and $\{ u_k\}$ such that $\| u_k\|_{C^0_\delta}=1$ while $\| \nabla_{\A'_{\epsilon_k}}^\ast\nabla_{\A'_{\epsilon_k}}u_k\|_{C^{0,\alpha}_{\delta-2}}\ra 0$.

The connection $\A'_\epsilon$ of \eqref{eq:Approximate:Solution} is obtained by gluing $\A_\tu{sing}$ and the connection $\A_\mu({\omega^i_\mu})
$ in an annulus of radius $r_j^i \sim R(\epsilon)$ with $\epsilon^{-1}R(\epsilon)\ra \infty$ as $\epsilon\ra 0$. From Lemma \ref{lem:Singular:Caloron:Holonomy:Estimate}
 we also know that $\partial_t\lrcorner\A'_\epsilon$ lies in a compact subset of $\mathring{A}^+$ on regions where $r_\mu^i \geq \epsilon^\tau$ for all $\mu,i$ for any $\tau\in (0,1)$. We then have an analogue of \eqref{eq:Poincare:OffDiagonal:Oscillatory}: for sections supported in this region we have a decomposition $u=u_0^0 +u_0' + u_\perp$ and
\begin{equation}\label{eq:Approximate:Poincare:OffDiagonal:Oscillatory}
\| u_0' + u_\perp \|_{C^0_\delta} \leq C\epsilon^{1-\tau}\|\nabla_{\A'_\epsilon}u\|_{C^{0,\alpha}_{\delta-1}}.
\end{equation}
We therefore conclude that on regions where $r_\mu^i \geq  \epsilon^\tau$ for all $\mu$ and $i$, $u_k$ converges to an $S^1$--invariant ``diagonal'' section (\ie an $\Lie{h}$--valued function) $u_\infty$ on $\R^3\setminus \bigcup_{\mu,i}\{ p_\mu^i\}$, which must be harmonic and satisfies $|u_\infty|\leq C (r_\mu^i)^\delta$ near $p_\mu^i$ and $|u_\infty|\leq C r^\delta$ as $r\ra \infty$. Since $\delta>-1$, $u_\infty$ in fact extends to a harmonic function on the whole of $\R^3$ and since $\delta<0$ it must decay at infinity: it then vanishes by the maximum principle.

We therefore conclude that there exists points $x_k\ra \{ p_\mu^i\}\times S^1$ for some $\mu,i$ such that $r_{\epsilon_k}^{-\delta}|u_k (x_k)|\geq c >0$ for some $c$. We now rescale around $p_\mu^i$ by $\epsilon_k$ so that we reduce to work on $(\R^3\times S^1,g_1)$ with a sequence of connections $\A''_k$ converging in $C^{\infty}_{-1-\nu}$, $\nu >0$, to the fundamental caloron $\A_\mu (\omega)$ and a sequence $u_k$ uniformly bounded in $C^{2,\alpha}_\delta$ with $|u_k (x_k)|\geq c \left( 1+|\pi_{\R^3}(x_k)|^2\right)^{\frac{\delta}{2}}$ and $\nabla_{\A''_k}^\ast\nabla_{\A''_k}u_k \ra 0$ in $C^{0,\alpha}_{\delta-2}$. Here by abuse of notation we denote by the same symbols the sections $u_k$ and the points $x_k$ before and after rescaling.

By the Arzel\`a--Ascoli Theorem, after passing to a subsequence, $u_k$ converges to an element $u_\infty$ in the kernel of the Bochner Laplacian of $\A_\mu (\omega)$ in $C^{0}_{\delta}$, which must vanish by Proposition \ref{prop:Caloron:Isomorphism}. It follows that the points $x_k$ must satisfy $|\pi_{\R^3}(x_k)|=:R_k\ra \infty$. In order to get a contradiction, we now blow down $(\R^3\times S^1,g_1,\A''_k, u_k)$ by $R_k^{-1}$. Now note that, away from a compact set of $\R^3\times S^1$, $\partial_t\lrcorner\A''_k$ lies in a compact subset of $\mathring{A}^+$. Thus after rescaling we have a decomposition $u=u_0^0 + u_0'+u_\perp$ and an estimate \eqref{eq:Poincare:OffDiagonal:Oscillatory}
\[
\| u_0' + u_\perp \|_{C^0_\delta} \leq C R_k^{-(1-\tau)}\|\nabla_{\A''_k}u\|_{C^{0,\alpha}_{\delta-1}}.
\]
outside any ball of radius $R_k^{-\tau}$, $\tau\in (0,1)$, centred at the origin. We deduce that after rescaling $u_k$ converges to an $\Lie{h}$--valued harmonic function $u_\infty$ on $\R^3\setminus\{ 0\}$ satisfying $|u_\infty|\leq C r^\delta$ and $|u_\infty (x_\infty)|\geq c>0$ at some point $x_\infty\in S^2\subset \R^3$. Since $\delta \in (-1,0)$, the growth condition forces $u_\infty$ to be a decaying harmonic function on $\R^3$. It must therefore vanish, contradicting the existence of $x_\infty$.
\endproof
\end{prop}

In the last section of the paper we will need an additional estimate for the first order operator $D=D_{\A'_\epsilon}= d^\ast_{\A'_\epsilon}\oplus d^+_{\A'_\epsilon}\co C^{1,\alpha}_{\delta-1}\ra C^{0,\alpha}_{\delta-2}$ with $\delta\in (-1,0)$. By Proposition \ref{prop:Caloron:Fredholm:Dirac} $D$ is Fredholm and surjective. Moreover, Proposition \ref{prop:Linearised} implies that $D$ has a right inverse $G$ with uniformly bounded norm independently of $\epsilon>0$. Indeed, in the next section we will show that $DD^\ast$ is an arbitrarily small perturbation of $\nabla_{A'_\epsilon}^\ast\nabla_{\A'_\epsilon}$ as $\epsilon\ra 0$, so that Proposition \ref{prop:Linearised} allows one to define $G=D^\ast (DD^\ast)^{-1}$ with the claimed uniform estimate. The additional estimate we will need establishes the concentration of elements in the kernel of $D$ near the gluing regions in a uniform quantitative sense as $\epsilon \ra 0$.

\begin{prop}\label{prop:Approximate:Dirac:Kernel}
Fix $\delta\in (-1,0)$, $\alpha\in (0,1)$, a closed subset $\Omega$ of $(\R^3\setminus \{ p^i_\mu \})\times S^1$ and $\eta>0$. Then there exists $\epsilon_0$ such that for every $\epsilon\in (0,\epsilon_0)$ and $\xi\in C^{1,\alpha}_{\delta-1}$ with $D\xi =0$, we have
\[
\|\xi\|_{C^{0,\alpha}_{\delta-1}(\Omega)} \leq \eta \| \xi \|_{C^{1,\alpha}_{\delta-1}}.
\]
\proof
Assume by contradiction that there exists $\eta_0>0$, a sequence $\epsilon_k\ra 0$ and elements $\xi_k$ in the kernel of $D_k = D_{\A'_{\epsilon_k}}$ such that
\[
\|\xi_k\|_{C^{1,\alpha}_{\delta-1}}=1, \qquad \|\xi_k\|_{C^{0,\alpha}_{\delta-1}(\Omega)}\geq \eta_0.
\]

For $k$ sufficiently large, we can assume that $\Omega$ is contained in the region where $r_\mu^i \geq \epsilon_k^\tau$ for all $\mu,i$ and some $\tau\in (0,1)$. In particular, $\A'_{\epsilon_k}$ is abelian over $\Omega$. Using the trivialisation of $T^\ast (\R^3\times S^1)$ by orthonormal parallel $1$-forms, we can decompose the $1$-form $\xi_k$ as $\xi_k = (\xi_k)^0_0 + (\xi_k)_0'+(\xi_k)_\perp$ and we have the strong estimate \eqref{eq:Approximate:Poincare:OffDiagonal:Oscillatory} that implies
\[
\|(\xi_k)_0' + (\xi_k)_\perp\|_{C^{0,\alpha}_{\delta-1}(\Omega)}\leq C\epsilon_k^{1-\tau}, \qquad \|(\xi_k)^0_0\|_{C^{0,\alpha}_{\delta-1}(\Omega)}\geq \tfrac{1}{2}\eta_0.
\]
Writing $\xi_k = a_k + \epsilon_k\, \psi_k\, dt$, we conclude that, after passing to a subsequence $(a_k,\psi_k)$ converges to a non-trivial $S^1$--invariant ``diagonal'' (\ie $\Lie{h}$--valued) pair $(a_\infty,\psi_\infty)$ of a $1$-form $a_\infty$ and function $\psi_\infty$ on $\R^3\setminus \bigcup_{\mu,i}\{ p_\mu^i\}$, which satisfies the first order system
\[
\ast da_\infty - d\psi_\infty =0 = d^\ast a_\infty
\]
and the growth conditions $|(a_\infty,\psi_\infty)|\leq C (r_\mu^i)^{\delta-1}$ near $p_\mu^i$ and $|(a_\infty,\psi_\infty)|\leq C r^{\delta-1}$ as $r\ra \infty$. In particular, the coefficients of $a_\infty$ in a parallel trivialisation of $T^\ast \R^3$ and $\psi_\infty$ are decaying harmonic functions with controlled blow-up rate at each of the punctures. Since $\delta-1>-2$, the only singularity allowed at each puncture is the Green's function singularity. It follows that there exist constants $(a^i_\mu,\psi^i_\mu)\in \R^3\times \R$ such that
\[
a_\infty = dx \cdot \sum_{\mu,i}{\tfrac{a^i_\mu}{|x-p_\mu^i|}}, \qquad \psi_\infty = \sum_{\mu,i}{\tfrac{\psi^i_\mu}{|x-p_\mu^i|}}.
\] 
However, it is easy to see that these are not solutions of the first order system satisfied by $(a_\infty,\psi_\infty)$ unless $(a^i_\mu,\psi^i_\mu)=0$ for all $\mu, i$, therefore reaching a contradiction. Indeed, one calculates that the first order system is equivalent to
\[
 \sum_{\mu,i}{\tfrac{\psi^i_\mu x - a^i_\mu \times x}{|x-p_\mu^i|^3}} = 0 = x \cdot \sum_{\mu,i}{\tfrac{a^i_\mu}{|x-p_\mu^i|}},
 \] 
where ${}\cdot{}$ and ${}\times{}$ denote the dot and cross products in $\R^3$.
\endproof
\end{prop}

Finally, we also note that we have an analogue of Proposition \ref{prop:Approximate:Schauder}
\begin{prop}\label{prop:Approximate:Dirac:Schauder}
Given $\delta\in \R$, there exists a constant $C$ independent of $\epsilon$ such that
\[
\| \xi\|_{C^{1,\alpha}_{\delta-1}} \leq C \left( \|D_{\A'_\epsilon} \xi \|_{C^{0,\alpha}_{\delta-2}} + \| \xi \|_{C^0_{\delta-1}}\right)
\]	
for all $\xi\in C^{1,\alpha}_{\delta-1}$.
\end{prop}

\section{Existence}
\label{sec:existence}

Armed with the results of the last section we now return to the approximate caloron $\A_{\epsilon}'$ of \eqref{eq:Approximate:Solution}. We want to prove the existence of a ``small'' $1$--form $a$ with values in the adjoint bundle such that $\A_{\epsilon}'+a$ is a caloron. In order to take into account the invariance of the anti-self-duality equations under gauge tranformations, we will look for $a$ of the form $a=d_{\A'_\epsilon}^\ast u$ for a self-dual $2$-form $u$ with values in the adjoint bundle. We will prove the existence of $u$ using the following quantitative version of the Implicit Function Theorem.

\begin{lemma}\label{lem:IFT}
Let $\Phi\co E\ra F$ be the smooth function between Banach spaces and write $\Phi (x)=\Phi(0) + L(x)+N(x)$, where $L$ is linear and $N$ contains the non-linearities. Assume that there exists constants $C_1, C_2,C_3$ such that
\begin{enumerate}
\item $L$ is invertible with $\| L^{-1} \| \leq C_1$;
\item $\| N(x)-N(y)\|_F \leq C_2 \|x+y\| _E \| x-y\| _E$ for all $x,y\in B_{C_3}(0)\subset E$;
\item $\| \Phi(0)\| _{F}< \min\left\{ \frac{C_3}{2C_1},\frac{1}{4C_1^2 C_2}\right\}$.
\end{enumerate}
Then there exist a unique $x\in E$ with $\| x \| _E \leq 2C_1\| \Phi(0)\|_F$ such that $\Phi(x)=0$.
\end{lemma}

In our situation, we fix $\delta\in (-1,0)$, $\alpha\in (0,1)$ and we set
\[
E:=C^{2,\alpha}_\delta\Omega^+ (\R^3\times S^1;\tu{ad}\, P),\qquad 
F:=C^{0,\alpha}_{\delta-2}\Omega^+(\R^3\times S^1;\tu{ad}\, P).
\]
For $u\in E$ we let $\Phi(u)$ denote the self-dual part of the curvature of $\A'_\epsilon + d_{\A'_\epsilon}^\ast u$, so that the decomposition $\Phi(u)=\Phi(0) + L(u) + N(u)$ reads
\[
\Phi (u) = F_{\A'_\epsilon}^+ + d_{\A'_\epsilon}^+ d_{\A'_\epsilon}^\ast u + \tfrac{1}{2}[d_{\A'_\epsilon}^\ast u,d_{\A'_\epsilon}^\ast u]^+.
\]

We need to check that the hypotheses of Lemma \ref{lem:IFT} are satisfied.

\subsubsection*{The linear term}

As $u$ is a self-dual $2$-form with values in the adjoint bundle we can write $u=u_{1}\,\omega_{1}+u_{2}\,\omega_{2}+u_{3}\,\omega_{3}$ for the hyperk\"ahler triple $(\omega_1,\omega_2,\omega_3)$ inducing $g_\epsilon$ and sections $u_1,u_2,u_3$ of $\tu{ad}\, P$. Since the hyperk\"ahler triple is parallel, the Weitzenb\"{o}ck identity \eqref{eq:Weitzenbock} yields 
\begin{equation}
d_{\A_{\epsilon}'}d_{\A_{\epsilon}'}^{*}u=\sum_{a=1}^{3}\left(\nabla_{\A_{\epsilon}'}^{*}\nabla_{\A_{\epsilon}'}u_{a}\right)\omega_{a}+F^{+}_{\A_{\epsilon}'}\cdot u.
\end{equation}
In order to reduce part (i) of Lemma \ref{lem:IFT} with an $\epsilon$--independent constant $C_1$ from Proposition \ref{prop:Linearised} we only need to check that the last term can be regarded as a small perturbation. For this, observe that
\[
\| F^{+}_{\A_{\epsilon}'}\cdot u\| _{C^{0,\alpha}_{\delta-2}}\leq C \| F_{\A'_\epsilon}^+\|_{C^{0,\alpha}_{-2}}\| u \|_{C^{0,\alpha}_{\delta}}
\]
for a uniform constant $C$ and $\| F_{\A'_\epsilon}^+\|_{C^{0,\alpha}_{-2}} = O(\epsilon^2 |\ln \epsilon|^2 )$ by Lemma \ref{lemma:bounds on Fplus}.

\subsubsection*{The non-linear term}

It is clear that the non-linear term $\tfrac{1}{2}[d_{\A'_\epsilon}^\ast u,d_{\A'_\epsilon}^\ast u]^+$ is controlled by the norm of the multiplication map
\[
C^{1,\alpha}_{\delta-1}\times C^{1,\alpha}_{\delta-1}\ra C^{0,\alpha}_{\delta-2},
\]
which in turn is easily seen to be controlled by $\| r_\epsilon^\delta\|_{L^\infty}$, \cf \eqref{eq:Product:Weighted:Holder}. Since $\delta<0$ and $r_\epsilon\geq \epsilon$ by Definition \ref{def:Approximate:Weight} we conclude that part (ii) of Lemma \ref{lem:IFT} holds with a constant $C_2 = O(\epsilon^\delta)$.

\subsubsection*{The error}

By Lemma \ref{lemma:bounds on Fplus} we have $\| F_{\A'_\epsilon}^+\|_{C^{0,\alpha}_{\delta-2}} \leq C (\epsilon |\ln \epsilon|)^{2-\delta}$, so that
\[
C_2 \| F_{\A'_\epsilon}^+\|_{C^{0,\alpha}_{\delta-2}} \propto \epsilon^\delta \| F_{\A'_\epsilon}^+\|_{C^{0,\alpha}_{\delta-2}} = O(\epsilon^2 |\ln \epsilon|^{2-\delta})
\]
can be made arbitrarily small as $\epsilon\ra 0$.

\subsubsection*{The existence theorem}
An application of Lemma \ref{lem:IFT} now allows us to deform the approximate caloron $\A'_\epsilon$ of \eqref{eq:Approximate:Solution} to an exact self-dual connection.

\begin{theorem}\label{thm:Existence}
Fix a semisimple structure group $G$, holonomy parameter $\omega\in \mathring{A}^{+}$, instanton number $n_0$ and total magnetic charge $\gamma_\tu{m}=\sum_{\mu=0}^{\tu{rk}}{n_\mu\,\alpha^\vee_\mu}$. If $n_\mu\geq 0$ for all $\mu=0,\dots, \tu{rk}$, then the moduli space of calorons $\mathcal{M}_\epsilon (\omega,\gamma_{\tu{m}},n_0)$ is non-empty.

More precisely, for any choice of $n_0+\dots + n_{\tu{rk}}$ distinct points $p^{1}_{0},\dots, p^{n_0}_{0},\dots, p^{1}_{\rm rk},\dots, p^{n_{\rm rk}}_{\rm rk}$ in $\R^3$ and phases $\psi^{1}_{0},\dots, \psi^{n_{\rm rk}}_{\rm rk}\in \unitary{1}$ there exists $\epsilon_0>0$ (uniform in the minimum distance $d_{\min}$ between the distinct points) such that for all $\epsilon\in (0,\epsilon_0)$ there is a caloron $\A_\epsilon$ in $\mathcal{M}_\epsilon (\omega,\gamma_{\tu{m}},n_0)$ with the following behaviour as $\epsilon\ra 0$:
\begin{itemize}
\item[(i)] $\A_{\epsilon}$ smoothly converges to the flat abelian connection $d+\omega \otimes dt$ on compact subsets of $\left(\R^{3}\backslash\{p_{0}^{1},\dots, p_{\rm rk}^{n_{\rm rk}}\}\right)\times S^{1}$;
\item[(ii)] after rescaling by $\epsilon^{-1}$ near a point $p_{\mu}^{i}$, $\A_{\epsilon}$ smoothly converges to the fundamental caloron $\A_{\mu}\left(\omega\right)$ on compact subsets of $\R^3\times S^1$.
\end{itemize} 
\proof
The discussion so far implies that the family of approximate calorons $\A'_\epsilon$ constructed in \eqref{eq:Approximate:Solution} can be deformed to an exact solution $\A_\epsilon = \A'_\epsilon + d^\ast_{\A'_\epsilon}u_\epsilon$ for all $\epsilon$ sufficiently small. Indeed, Lemma \ref{lem:IFT} guarantees the existence of $u_\epsilon \in C^{2,\alpha}_\delta$ with $\| u_\epsilon \|_{C^{2,\alpha}_\delta}=O((\epsilon |\ln \epsilon|)^{2-\delta})$ for all $\epsilon$ sufficiently small.

The limiting properties (i) and (ii) are satisfied by $\A'_\epsilon$ by direct inspection. Moreover, since $R(\epsilon)\ra 0$, given any compact set $K$ of $\left(\R^{3}\backslash\{p_{0}^{1},\dots, p_{\rm rk}^{n_{\rm rk}}\}\right)\times S^{1}$ we can assume that $\A'_\epsilon$ is self-dual on $K$ for all small enough $\epsilon$. The uniform weighted elliptic estimates of Remark \ref{rmk:Approximate:Schauder} applied to $u_\epsilon|_K$ then yield (i) for $\A_\epsilon$ as well. Part (ii) is obtained in a similar way using the fact that $\epsilon^{-1}R(\epsilon)\ra \infty$.

Finally, non-emptyness of the moduli space for all $\epsilon>0$ follows by scaling. 
\endproof
\end{theorem}

\begin{remark*} While $\A_{\epsilon}$ tends to the vacuum connection away from $\{ p_{\mu}^{i}\}_{\mu,i}$ as $\epsilon\to 0$, the curvature $F_{\A_{\epsilon}}$ diverges at those points. This is because $\A_{\epsilon}$ approximates the singular connection $\A_{\text{sing}}$ in the limit $\epsilon\to 0$.
\end{remark*}
\begin{remark*}
The phases $\psi^{i}_{\mu}$ in the statement of the Theorem are the gluing parameters of Remark \ref{rmk:G caloron asymptotics} which line up the framing of $\A_{\text{sing}}$ and $\A_{\mu}$ on $U_{\text{sing}}\cap U^{i}_{\mu}$.
\end{remark*}

\begin{remark*}
A natural expectation is that $\mathcal{M}_\epsilon (\omega,\gamma_\tu{m},n_0)$ is non-empty if and only if $n_\mu \geq 0$ for all $\mu =0,\dots, \tu{rk}$ and reduces to a single point (a flat connection) if $n_0 = \dots = n_\tu{rk} =0$. This is known for $G=SU(n)$ or when the instanton number is $n_0=0$, in which case calorons reduce to monopoles on $\R^3$ by Proposition \ref{prop:Circle:Invariant:Calorons} below. There are three ideas for how one might prove the statement for general $G$ and $n_0$. 
\begin{enumerate}
\item[(i)] In the case of $G=SU(n)$, the inequalities $n_\mu\geq 0$ are proved by considering the index for a family of twisted Dirac operators. More precisely, given an $SU(n)$ caloron $\A$, consider the family of unitary connections $\A_s = \A + i s\, \tu{id}_n$ on the rank-$n$ complex vector bundle $E$ associated with the standard representation of $SU(n)$. The Dirac operator $D_{\A_s}$ on $E$ is Fredholm and surjective for generic $s$, \cf Proposition \ref{prop:Caloron:Fredholm:Dirac}. Hence $\tu{index}(D_{\A_s})\geq 0$ for any such $s$. These indices are related to the constituent monopole numbers $n_{\mu}$ and varying $s$ shows that $n_{\mu}\geq 0$ for all $\mu=0, \dots, {\rm rk}$.

The idea is to do the same for a general compact semi-simple group $G$ by embedding it in $SU(N)$ via a representation $\rho:G\to SU(N)$. Such a representation is described in terms of its weights: given $H\in\Lie{h}$ we have
\begin{equation*}
\rho\left(\exp\left(H\right)\right)=\text{diag}\left\{e^{iw\left(H\right)}\vert w\in \hat{V}\right\},
\end{equation*}
where $\hat{V}$ is the set of weights of $\rho$. Any $G$ caloron then induces an $SU(N)$ caloron $\A_{\rho}$ and as before a family of generically surjective Fredholm Dirac operators $D_{\A_{\rho,s}}$. Computing their indices---an example of how to do this in detail using \cite[Theorem D]{Cherkis_2021} is given below in the proof of Proposition \ref{prop:Fundamental:Caloron:Dimension:Transverse}---yields
\[
\text{ind}\left(D_{\A_{\rho,s}}\right)=\sum_{w\in\hat{V}}\lfloor w(\omega)\rfloor w\left(\gamma_\tu{m}\right)+n_{0}\,\text{ind}_{D}\left(\rho\right)+\sum_{w\vert s_{w}<s}w\left(\gamma_\tu{m}\right) \geq 0,
\]
where $\lfloor w(\omega)\rfloor$ is the largest integer smaller or equal than $w(\omega)$, $\{ w(\omega)\} = w(\omega) - \lfloor w(\omega)\rfloor$ and $s_{w}=1-\{w(\omega)\}\in(0,1]$. Moreover, $\text{ind}_{D}\left(\rho\right)$ is the Dynkin index of $\rho$, \cf the proof of Proposition \ref{prop:Fundamental:Caloron:Dimension:Transverse}.

Although we were unable to work out the combinatorics, it seems likely that by varying the representation $\rho$ and $s$ we would recover $n_{\mu}\geq 0$ for all $\mu$. For example, when $\rho=\Ad$ is the adjoint representation we have
\[
\text{ind}\left(D_{\A_{\Ad,s}}\right)=2\sum_{\mu=0}^{\rm rk}n_{\mu}+\sum_{\alpha\in R^{+}}\left(\delta_{s^{+}_{\alpha}}-\delta_{s^{-}_{\alpha}}\right)\alpha\left(\gamma_\tu{m}\right) \geq 0,
\]
with $s_{\alpha}^{+}=1-\alpha(\omega)$, $s_{\alpha}^{-}=\alpha(\omega)$ for all positive roots $\alpha$ and
\begin{equation*}
\delta_{s^{\pm}_{\alpha}}=\begin{cases}
1\quad s>s^{\pm}_{\alpha},\\
0 \quad s<s^{\pm}_{\alpha}.
\end{cases}
\end{equation*}
Now, if $\omega$ is such that $\alpha_{0}(\omega)>-\frac{1}{2}$ we can find generic $s$ such that $1-\alpha(\omega)>s>\alpha(\omega)$ for all $ \alpha\in R^{+}$. Using $\sum_{\alpha\in R^+}{\alpha(\gamma_\tu{m})}= 2 \sum_{\mu=1}^{\tu{rk}}{n_\mu - n_0\, m_\mu}$ and $m_\mu\geq 0$ we deduce that $n_{0}\geq 0$ in this particular case.
 
 \item[(ii)] When $n_0=0$ and $\A$ is a monopole the integers $n_1,\dots, n_\tu{rk}$ have an interpretations in terms of based rational maps from $\PP^1$ into the flag manifold $G/T$ as the (necessarily non-negative) degrees of the pull-backs of the ample line bundles on the flag manifold associated with the fundamental weights of $G$ \cite{Jarvis:Rational:Maps:I,Jarvis:Rational:Maps:II}. An extension of this argument to calorons requires one to work with rational maps into infinite dimensional flag varieties associated with loop groups \cite{Norbury:Loop:Group}.
 
\item[(iii)] Finally, we propose a more analytic approach that uses the formula \eqref{eq:caloron energy expression} for the Yang--Mills energy of a caloron and the persistence of solutions as we vary the holonomy parameter $\omega\in \mathring{A}^+$. A natural expectation is that for fixed $\gamma_{\tu{m}}\in \Lambda$ and $n_{0}$ the set
\begin{equation*}
\{\omega\in\mathring{A}^{+}\, \vert \,\mathcal{M}_{\epsilon}\left(\omega,\gamma_{\tu{m}},n_{0}\right)\neq \emptyset\}
\end{equation*}
is either empty or the whole of $\mathring{A}^{+}$. Openness of this set in $\mathring{A}^+$ is easily established using the analytic results of Section \ref{sec:Analysis}, but closedness appears more challenging because of non-compactness phenomena such as instanton bubbling. If the claim were true, then one could assume by contradiction that there exists $\mu$ such that $n_{\mu}<0$, say $n_{0}<0$. Then one could take $\omega\in \mathring{A}^{+}$ sufficiently close to 0 to deduce from \eqref{eq:caloron energy expression} that for a putative caloron $\A$ in $\mathcal{M}_{\epsilon}\left(\omega,\gamma_{m},n_{0}\right)$ we would have $\mathcal{YM}(\A)\approx n_{0}<0$ and therefore reach a contradiction.
\end{enumerate}

\end{remark*}

\section{Index Computations}\label{sec:Index Computations}

Fix $\epsilon>0$, holonomy parameter $\omega\in \mathring{A}^+$, magnetic charge $\gamma_{\tu{m}}\in \Lambda$ and instanton number $n_0\in\Z$ and consider the corresponding moduli space $\mathcal{M}_\epsilon (\omega,\gamma_\tu{m},n_0)$ of (framed) calorons. The analytic results of Section \ref{sec:Analysis}, in particular Propositions \ref{prop:Caloron:Isomorphism} and \ref{prop:Caloron:Fredholm:Dirac}, imply in a standard way that, fixing $\alpha\in (0,1)$ and $\delta\in (-1,0)$, $\mathcal{M}_\epsilon (\omega,\gamma_\tu{m},n_0)$ is a smooth (possibly empty) manifold with smooth structure induced by the Banach manifold structure on the space of connections of class $C^{1,\alpha}_{\delta-1}$ with fixed asymptotic model $\A_\infty(\omega,\gamma_\tu{m})$, acted upon by the group of gauge transformations of class $C^{2,\alpha}_{\delta}$. Moreover, the equivalence of the $C^{1,\alpha}_{\delta-1}$ and $L^2$ kernels of the deformation operator $D_\A$ implies that $\mathcal{M}_\epsilon (\omega,\gamma_\tu{m},n_0)$ carries a natural Riemannian metric arising from the $L^2$--inner product of infinitesimal deformations. This metric is hyperk\"ahler by virtue of an infinite dimensional hyperk\"ahler quotient (and is in general incomplete because of instanton bubbling).

In this section we calculate the dimension of $\mathcal{M}_\epsilon (\omega,\gamma_\tu{m},n_0)$, thus implying that the family of solutions produced by Theorem \ref{thm:Existence} depends on a full dimensional family of parameters.

\begin{remark*}
In the following we will primarily consider the deformation operator $D$	 as the (surjective) Fredholm operator $D\co C^
{1,\alpha}_{\delta-1}\ra C^{0,\alpha}_{\delta-2}$ with $\delta\in (-1,0)$. However, in the proof of Proposition \ref{prop:Fundamental:Caloron:Dimension:Transverse} we will apply an $L^2$--index formula that is justified in view of Proposition \ref{prop:Caloron:Fredholm:Dirac}, and in the proof of Theorem \ref{thm:Dimension} we will use the $\epsilon$--dependent norms of Definition \ref{def:Approximate:Holder}, which are equivalent to the $C^{k,\alpha}_\nu$--norms for any fix $\epsilon>0$.
\end{remark*}

\subsection{Moduli space of fundamental calorons}

In this section we show that the fundamental calorons of \eqref{eq:Fundamental Calorons} are indeed ``fundamental'', \ie they move in a $4$-dimensional moduli space. Since moduli spaces of calorons are hyperk\"ahler this is the lowest possible dimension of a non-trivial moduli space. In a physics context, this observation was made in \cite{Lee2_1998}. Here me make this more precise from a mathematical perspective.

\begin{theorem}\label{thm:Dimension:Fundamental:Calorons}
The dimension of the moduli space of the fundamental calorons are
\[
\tu{dim}\,\mathcal{M}_{\epsilon}\left(\omega,\alpha_{0}^{\vee},1\right)=4, \qquad \tu{dim}\,\mathcal{M}_{\epsilon}\left(\omega,\alpha_{\mu}^{\vee},0\right)=4,\quad \mu=1,\dots \text{rk}.
\]
\end{theorem}

The proof of the theorem takes the rest of this subsection. A fundamental $G$--caloron $\A_\mu (\omega)$ is reducible since it has structure group $T'\times SU(2)\subset G$. The deformation theory therefore splits into two independent contributions: deformations of $\A_\mu (\omega)$ as a $T'\times SU(2)$--caloron and deformations arising from the complement $\Lie{p}_\mu$ in the splitting $\Lie{g} = (\rho_\mu)_\ast (\Lie{h}'\oplus \Lie{su}_2)\oplus \Lie{p}_\mu$. We will show that deformations of $\A_\mu (\omega)$ as a $G$--caloron in fact arise from deformations of the charge 1 BPS monopole, \ie there are no deformations of $\A_\mu (\omega)$ arising from $\Lie{p}_\mu$ and there are no unexpected deformation of the fundamental $SU(2)$ calorons $\A^\pm_{\rm BPS}$.

\subsubsection{Deformations as a $T'\times SU(2)$--caloron}

We first consider the deformations of $\A_\mu (\omega)$ as a $T'\times SU(2)$--caloron. In other words, we want to consider moduli spaces of abelian calorons and the two moduli spaces $\mathcal{M}_{\epsilon}^{SU(2)}(\omega,1,0)$ and $\mathcal{M}_{\epsilon}^{SU(2)}(\omega,-1,1)$.

An abelian caloron $\A$ is uniquely determined up to gauge by its curvature $d\A$, which must be a closed anti-self-dual $2$-form, and therefore also coclosed. Moreover, the boundary conditions \eqref{eq:Framing} imply that $d\A$ is $L^2$--integrable. $L^2$ Hodge theory on ALF spaces \cite[Corollary 1 and \S 7.1.2]{HHM} yields immediately (since the compactification of $\R^3\times S^1$ relevant to the work of Hausel--Hunsicker--Mazzeo is $S^4$) that any abelian caloron is flat and therefore uniquely determined by the holonomy parameter.

The moduli spaces $\mathcal{M}_{\epsilon}^{SU(2)}(\bar{\omega},1,0)$ and $\mathcal{M}_{\epsilon}^{SU(2)}(\omega,-1,1)$ are identified via the ``large'' gauge transformation \eqref{eq:rotation map large gauge transformation}, \ie via the ``rotation map'' of \cite[\S 2.2]{Nye_thesis}. The following result, whose proof is inspired by \cite[\S 3]{YWang}, implies that any caloron with vanishing instanton number is gauge equivalent to a monopole. As a consequence, $\mathcal{M}_{\epsilon}^{SU(2)}(\omega,1,0)$, $\mathcal{M}_{\epsilon}^{SU(2)}(\omega,-1,1)\simeq \mathcal{M}_{\epsilon}^{SU(2)}(\bar{\omega},1,0)$ are both diffeomorphic to $\R^3\times S^1$, the moduli space of charge $1$ monopoles.

\begin{prop}\label{prop:Circle:Invariant:Calorons}
Let $\A$ be a framed caloron with instanton number $n_0=0$. Then there exists a gauge transformation compatible with the framing such that $u^\ast\A$ is the pull-back of a monopole.
\proof
Represent $\A$ by a connection $\A = A_t + \epsilon\, \Phi_t\otimes dt$ on a bundle $P_h= (\R^3\times\R\times G)/\Z$ for some map $h\co \R^3\ra G$ with $\deg{h}=0$. We can easily reduce to the case $h=\tu{id}$, \cf \cite[Lemma 2.20]{Nye_thesis}, so we assume without loss of generality that $(A_t,\Phi_t)$ is a 1-parameter family of pairs of a connection and Higgs field on $\R^3$ which is periodic in $t$ with period $2\pi$ and satisfies $\partial_t A_t = \ast_{\R^3} F_{A_t} - d_{A_t}\Phi_t$.

Firstly, by Proposition \ref{prop:Caloron:Isomorphism} applied to $t$--independent sections, for each $t$ we can find a gauge transformation $u_t\co \R^3\ra G$ which decays to the identity at infinity such that the curve $(A'_t,\Phi'_t)=u_t^\ast (A_t,\Phi_t)$ satisfies $d_{A'_t}^\ast (\partial_t A'_t)-[\Phi'_t,\partial_t\Phi'_t]=0$, \ie for each $t$ the infinitesimal variation $\partial_t (A'_t,\Phi'_t)$ satisfies a natural gauge fixing condition with respect to $(A'_t,\Phi'_t)$. Note that $(A'_t,\Phi'_t)$ is not necessarily periodic anymore. However, $(A'_t,\Phi'_t)$ are all asymptotic to a fixed periodic asymptotic model which satisfies the Bogomolny equation. In particular,
\begin{equation}\label{eq:Circle:Invariant:Calorons1}
A'_t - A'_0=O(r^{-1-\delta}),\qquad \partial_t (A'_t,\Phi'_t) = O(r^{-1-\delta}) 
\end{equation}
for some $\delta>0$ and the boundary terms
\begin{equation}\label{eq:Matching boundary terms}
\lim_{r\ra \infty}{\int_{\partial B_r}{\langle \Phi'_t \wedge F_{A_t}\rangle}} = \lim_{r\ra \infty}{\int_{\partial B_r}{\langle \Phi'_t\wedge \ast_{\R^{3}}d_{A_t}\Phi'_t\rangle}} 
\end{equation}
are well defined. Finally, we still have
\begin{equation}\label{eq:Circle:Invariant:Calorons3}
\partial_t A'_t = \ast_{\R^3}F_{A'_t}-d_{A'_t}\Phi'_t.
\end{equation}

We will now show that $\ast_{\R^3}F_{A'_t}-d_{A'_t}\Phi'_t =0$. Noting that
\[
\int_{\R^3}{|F_{A_t}|^2+ |d_{A'_t}\Phi'_t|^2} = \int_{\R^3}{|\ast_{\R^3}F_{A'_t}-d_{A'_t}\Phi'_t|^2} + 2 \lim_{r\ra \infty}{\int_{\partial B_r}{\langle \Phi'_t \wedge F_{\A_t}\rangle}},
\]
and taking in to account \eqref{eq:Matching boundary terms}, our claim will follow from showing separately that
\begin{equation}\label{eq:Circle:Invariant:Calorons2}
\int_{\R^3}{|F_{A_t}|^2} =\lim_{r\ra \infty}{\int_{\partial B_r}{\langle \Phi'_t \wedge F_{\A_t}\rangle}}, \qquad \int_{\R^3}{|d_{A'_t}\Phi'_t|^2} = \lim_{r\ra \infty}{\int_{\partial B_r}{\langle \Phi'_t \wedge \ast_{\R^{3}}d_{A_t}\Phi'_t\rangle}}.
\end{equation}
In order to show the first equality we argue as in \cite[\S 3]{YWang} using the variation for the Chern--Simons functional of $A'_t = A_0' + (A'_t - A'_0)$. Here recall that the Chern--Simons functional $CS_{A_{0}}\left(A_{0}+a\right)$ is defined as
\begin{equation*}
CS_{A_{0}}\left(A_{0}+a\right)=-\int_{\R^{3}}{\left\langle d_{A_{0}}a \wedge a +\frac{2}{3}a\wedge a\wedge a +2a\wedge F_{A_{0}}\right\rangle}.
\end{equation*}
Then using the gauge invariance of the Chern--Simons functional and the periodicity of $A_t$ we have  (see \eg (36) in \cite{YWang})

\begin{align*}
0 &= \tu{CS}_{A_0}(A_{2\pi})-\tu{CS}_{A_0}(A_0)=\tu{CS}_{A'_0}(A'_{2\pi})-\tu{CS}_{A'_0}(A'_0)\\
 &= \int_{0}^{2\pi}{\left( -2\int_{\R^3}{\langle \partial_t A'_t \wedge F_{A'_t}\rangle} + \lim_{r\ra \infty}{\int_{\partial B_r}{\langle \left(A'_t - A'_0\right)\wedge \partial_t A'_t\rangle}} \right)}.
\end{align*}

The boundary term vanishes by \eqref{eq:Circle:Invariant:Calorons1}, so the expression for $\partial_t A'_t$ in \eqref{eq:Circle:Invariant:Calorons3} and another integration by parts yield
\[
0 = \int_{\R^3}{\langle \partial_t A'_t\wedge F_{A'_t}\rangle} = \int_{\R^3}{|F_{A_t}|^2} -\lim_{r\ra \infty}{\int_{\partial B_r}{\langle \Phi'_t\wedge F_{\A_t}\rangle}}.
\]
In order to show the second equality in \eqref{eq:Circle:Invariant:Calorons2}, observe that
\[
d_{A'_t}^\ast d_{A'_t}\Phi'_t = - d_{A'_t}^\ast \left( \ast_{\R^3}F_{A'_t}-d_{A'_t}\Phi'_t\right) = - d_{A'_t}^\ast(\partial_t A'_t) = [\Phi'_t, \partial_t \Phi'_t]
\]
by \eqref{eq:Circle:Invariant:Calorons3} and our gauge-fixing condition. Hence $d_{A'_t}^\ast d_{A'_t}\Phi'_t$ is pointwise orthogonal to $\Phi'_t$ and an integration by parts yields
\[
0 = \langle d_{A'_t}^\ast d_{A'_t}\Phi'_t, \Phi'_t\rangle_{L^2} = \int_{\R^3}{|d_{A'_t}\Phi'_t|^2} - \lim_{r\ra \infty}{\int_{\partial B_r}{\langle \Phi'_t \wedge \ast_{\R^{3}}d_{A_t}\Phi'_t\rangle}}.
\]

We conclude that \eqref{eq:Circle:Invariant:Calorons2} holds and therefore $A'_t\equiv A'_0$ for all $t$ and $(A'_0,\Phi'_t)$ satisfies the Bogomolny equation. Differentiating the latter and using \eqref{eq:Circle:Invariant:Calorons1} we conclude that $\partial_t \Phi'_t$ is a decaying parallel section and therefore $\Phi'_t\equiv \Phi'_0$ for all $t$ also. Thus $(A_t,\Phi_t) = (u_t^{-1})^\ast (A'_0,\Phi'_0)$ and since $(A_t,\Phi_t)$ and $(A'_0,\Phi'_0)$ are both periodic and the stabiliser of $(A'_0,\Phi'_0)$ in the group of gauge transformations that decay to the identity is trivial, we deduce that $u_t$ is also periodic. 
\endproof 
\end{prop}

\subsubsection{Deformations as a $G$--caloron} We must now show that there are no additional deformations of $\A_\mu(\omega)$ as a $G$--caloron. Recall that we defined $\Lie{p}_\mu$ via the orthogonal splitting $\Lie{g} = (\rho_\mu)_\ast (\Lie{h}'\oplus \Lie{su}_2)\oplus \Lie{p}_\mu$. Note that $\Lie{p}_\mu$ is an orthogonal representation of $T'\times SU(2)$ so that we have an associated vector bundle $E_\mu\ra \R^3 \times S^1$ with fibre $\Lie{p}_\mu$ and a connection $\A$ on $E_\mu$ induced by $\A_\mu (\omega)$. Since the deformation operator $D_{\A}$ is surjective by Proposition \ref{prop:Caloron:Fredholm:Dirac}, the fact that there are no deformations of $\A_\mu(\omega)$ as a $G$--caloron that do not arise from deformations of $\A_\mu(\omega)$ as a $T'\times SU(2)$--caloron is equivalent to the following proposition.

\begin{prop}\label{prop:Fundamental:Caloron:Dimension:Transverse}
The index of $D_{\A}$ acting on $E_\mu$--valued $1$-forms vanishes.
\proof
First of all note that on any spin $4$-manifold $M$ we have $T^\ast M\otimes \C = S^+\otimes S^-$, where $S^\pm$ is the positive/negative spinor bundle. Moreover, if $M$ is hyperk\"ahler then $S^+$ is trivial, so that, denoting by $E_\mu^\C$ the complex vector bundle associated with the $T'\times SU(2)$--representation $\Lie{p}_\mu\otimes\C$,
\[
 \tu{index}(D_\A;E_\mu) = \tu{index}(D_\A;E^\C_\mu) = \tu{index}(D^-_\A;S^+\otimes E^\C_\mu) = 2\, \tu{index}(D^-_\A;E^\C_\mu),
\] 
where $\tu{index}(D^-_\A;E)$ denotes the index of the twisted Dirac operator $D^-_\A\co S^-\otimes E\ra S^+\otimes E$.

In order to calculate $\tu{index}(D^-_\A;E^\C_\mu)$ we use the index theorem of \cite[Theorem D]{Cherkis_2021} (note that \cite{Cherkis_2021} has the opposite orientation conventions of ours). Here we regard $\A_\mu (\omega)$ as a unitary connection $\A$ on $E^\C_\mu$. The index formula of \cite[Theorem D]{Cherkis_2021} involves the second Chern number $\frac{1}{8\pi^{2}}\int\Trace\left(F_{\A}\wedge F_{\A}\right)$ and a boundary term depending on the $\tu{rk}\, E^\C_\mu$ eigenvalues of the asymptotic holonomy of $\A$ and the $\tu{rk}\, E^\C_\mu$ integers determining the magnetic charge of $\A$.

Now, an explicit calculation as in \cite[\S 2.1.7]{Nye_thesis} shows that 
\[
-\frac{1}{8\pi^{2}}\int\Trace\left(F_{\A^+_{\rm BPS}}\wedge F_{\A^+_{\rm BPS}}\right) = 2\omega' , \qquad -\frac{1}{8\pi^{2}}\int\Trace\left(F_{\A^-_{\rm BPS}}\wedge F_{\A^-_{\rm BPS}}\right) = 1 - 2\omega'
\]
if the $SU(2)$ holonomy parameter is $\omega'\in (0,\tfrac{1}{2})$. We deduce that
\begin{subequations}\label{eq:Cherkis:Index:Chern}
\begin{equation}	
-\frac{1}{8\pi^{2}}\int\Trace\left(F_{\A}\wedge F_{\A}\right)=\text{ind}_{D}\left(\mathfrak{su}(2),\mathfrak{p}^{\C}_\mu\right)\left(n_0+\alpha_\mu\left(\omega\right)\right)
\end{equation}
with $n_0=1$ for $\mu =0$ and $n_0=0$ otherwise. Here $\text{ind}_{D}\left(\mathfrak{su}(2),\mathfrak{p}^{\C}_\mu\right)$ is the Dynkin index of the Lie algebra morphism $\mathfrak{su}(2)\ra \Lie{u}(\tu{dim}_\C\,\Lie{p}^\C)$, \ie the ratio between the pull-back of the Killing form of $\Lie{u}(\tu{dim}_\C\,\Lie{p}^\C)$ and the Killing form of $\Lie{su}_2$. With our conventions, the Killing form of $\Lie{u}(n)$ is given up to a sign by the trace of the product of two matrices, so that the positive coroot of $\Lie{su}_2$ has norm $\sqrt{2}$ and $\text{ind}_{D}\left(\mathfrak{su}(2),\mathfrak{p}^{\C}_\mu\right)$ is half the trace of the endomorphism $[\alpha^\vee_\mu,[\alpha^\vee_\mu,\,\cdot\,]]$ of $\Lie{p}^\C_\mu$. We calculate
\begin{equation}
\begin{split}
\text{ind}_{D}\left(\mathfrak{su}(2),\mathfrak{p}^{\C}_\mu\right)&= \tfrac{1}{2} \sum_{\alpha\neq \pm \alpha_\mu}{\alpha(\alpha^\vee_\mu)^2} = \sum_{\alpha\in R^+}{\alpha(\alpha_\mu^\vee)^2} -4=\tfrac{1}{2}\text{ind}_{D}\left(\mathfrak{g},\text{Ad}\right)\| \alpha^\vee_\mu\|^2_{\mathfrak{g}}-4.
\end{split}
\end{equation}
\end{subequations}
Here $\text{ind}_{D}\left(\mathfrak{g},\text{Ad}\right)$ is the Dynkin index of the Lie algebra homomorphism $\Lie{g}\ra \Lie{u}(\Lie{g}^\C)$ given by the adjoint representation. It can be shown \cite[Example 1.2]{Panyushev_2015} that
\begin{equation}\label{eq:Dynkin:Index:Adjoint}
\text{ind}_{D}\left(\mathfrak{g},\text{Ad}\right) = 2\left( 1-\rho(\alpha_0^\vee)\right),
\end{equation}
where $\rho = \tfrac{1}{2}\sum_{\alpha\in R^+}\alpha$ denotes the half-sum of positive roots.

In order to calculate the boundary term in the index theorem of \cite[Theorem D]{Cherkis_2021}, observe that $\Lie{p}_\mu\otimes\C = \bigoplus_{\alpha\neq \pm \alpha_\mu}\Lie{g}_\alpha$ in terms of the decomposition of $\Lie{g}\otimes \C$ into root spaces. The asymptotic form of $\A$ preserves this decomposition and the holonomy and magnetic charge of the line bundle arising from $\Lie{g}_\alpha$ for $\alpha\neq\pm\alpha_\mu$ are, respectively, $\alpha (\omega)$ and $\alpha (\alpha_\mu^\vee)$. Since $\alpha(\omega)\in (0,1)$ the asymptotic holonomy is non-trivial and therefore the application of \cite[Theorem D]{Cherkis_2021} is justified. Taking $k=0$ in the latter formula (since $\R^3\times S^1 = \tu{TN}_0$) and taking into account the different convention of magnetic charge, we calculate that the boundary term in the index formula of \cite[Theorem D]{Cherkis_2021} is
\begin{equation}\label{eq:Cherkis:Index:Boundary}
\begin{split}
2\sum_{\alpha\neq |\alpha_\mu|}\left(\tfrac{1}{2}-\alpha\left(\omega\right)\right)\alpha(\alpha_\mu^{\vee}) &= 2\sum_{\alpha\in R^+}\left(\tfrac{1}{2}-\alpha\left(\omega\right)\right)\alpha(\alpha_\mu^{\vee}) - 4\,\tu{sign}\,\alpha_\mu\,\left( \tfrac{1}{2}-|\alpha_\mu (\omega)|\right)\\
 &= 2\left( \rho (\alpha_\mu^\vee) - \tu{sign}\,\alpha_\mu\right) -\left( \tfrac{1}{2}\text{ind}_{D}\left( \mathfrak{g},\text{Ad}\right) \|\alpha_\mu^\vee\|_\Lie{g}^2 -4\right) \alpha_\mu (\omega) .
\end{split}
\end{equation}
Here $|\alpha_\mu| = (\tu{sign}\,\alpha_\mu)\, \alpha_\mu$ and we used
\[
2\sum_{\alpha\in R^+}{\alpha (\omega)\, \alpha (\alpha_\mu^\vee)} = \text{ind}_{D}\left(\mathfrak{g},\text{Ad}\right) \langle \omega,\alpha_\mu^\vee\rangle_{\Lie{g}} = \tfrac{1}{2}\text{ind}_{D}\left(\mathfrak{g},\text{Ad}\right)\| \alpha_\mu^\vee\|_{\Lie{g}}^2 \,\alpha_\mu (\omega).
\]

Putting together \eqref{eq:Cherkis:Index:Chern} and \eqref{eq:Cherkis:Index:Boundary} we obtain that
\[
\text{index}\left(D^-_{\A_{\rho}},E^{\C}_\mu\right)= \left( \tfrac{1}{2}\text{ind}_{D}\left( \mathfrak{g},\text{Ad}\right) \|\alpha_\mu^\vee\|_\Lie{g}^2 -4\right) n_0 + 2\left( \rho (\alpha_\mu^\vee) - \tu{sign}\,\alpha_\mu\right).
\]
If $\mu \neq 0$ then $\alpha_\mu>0$, $n_0=0$ and $\rho(\alpha_\mu^\vee)=1$. If $\mu=0$, then $\alpha_0<0$, $n_0=1$, $\|\alpha_0^\vee\|^2_\Lie{g}=2$ (since $\alpha_0^\vee$ is a long coroot) and the vanishing of the index follows from \eqref{eq:Dynkin:Index:Adjoint}.
\endproof	
\end{prop}

The proof of Theorem \ref{thm:Dimension:Fundamental:Calorons} is now complete.

\subsection{Excision}

We will now use an excision argument inspired by the proof of the excision principle in \cite[\S 7.1]{Donaldson:Kronheimer} (\cf also \cite[\S 8]{Foscolo:Deformation}) to prove the main result of this section.

\begin{theorem}\label{thm:Dimension}
For $\omega\in \mathring{A}^+$ and $\gamma_\tu{m} = \sum_{\mu=0}^{\tu{rk}}{n_\mu \alpha_\mu^\vee}\in \Lambda$ with $n_\mu \geq 0$ for all $\mu$ we have
\[
\
\tu{dim}\, \mathcal{M}_{\epsilon}\left(\omega,\gamma_{\tu{m}},n_{0}\right) = 4(n_0 + n_1 + \dots + n_{\tu{rk}}).
\]
\end{theorem}
\proof
For a caloron $\A$ whose gauge equivalence class lies in $\mathcal{M}_{\epsilon}\left(\omega,\gamma_{m},n_{0}\right)$ the deformation operator $D_{\A}$ is surjective by Proposition \ref{prop:Caloron:Fredholm:Dirac} and therefore the theorem is equivalent to showing the index formula
\begin{equation}\label{eq:Dimension}
{\rm index}\, D_{\A}=4(n_0 + n_1 + \dots + n_{\tu{rk}}).
\end{equation}

First of all, note that by the compactness of the multiplication map in \eqref{eq:Comopact:Embedding:Weighted:Holder} the index remains unchanged if we replace $\A$ with the approximate caloron $\A'_\epsilon$ constructed in \eqref{eq:Approximate:Solution}. In fact, we consider a slight variation of the construction in \eqref{eq:Approximate:Solution} and assume that $\A$ is $S^1$--invariant and abelian outside of $\bigcup_{\mu,i}B_\epsilon(p_\mu^i)\times S^1$ and coincides with the model $\A_\infty(\omega^i_\mu,\alpha_\mu^\vee)$ on the ``annulus'' $\left(B_2(p^i_\mu)\setminus B_\epsilon(p^i_\mu)\right)\times S^1$. Similarly, we let $\A^i_\mu$ be a connection on $\R^3\times S^1$ with similar properties as $\A$ but constructed from the single point $p^i_\mu$ instead of the collection $\{ p^i_\mu\} _{\mu,i}$. Then $\A$ coincides with $\A^i_\mu$ in $B_2(p^i_\mu)\times S^1$. The idea is to show that the index of $D_{\A}$ coincides with the index of the disjoint union of the $D_{\A^i_\mu}$'s and then use Theorem \ref{thm:Dimension:Fundamental:Calorons} to conclude that the latter index equal the right-hand side of \eqref{eq:Dimension}. 

Now, $D_\A$ and $D_{\A^i_\mu}$ are all surjective with right inverses $G=D_{\A}^\ast (D_\A D_\A^\ast)^{-1}$ and similarly defined $G^i_\mu$ with uniformly bounded norms independently of $\epsilon$ by Proposition \ref{prop:Linearised}. Hence showing that
\[
{\rm index}\, D_{\A} = \sum_{\mu,i}{\rm index}\, D_{\A^i_\mu}
\] 
amounts to showing that the finite dimensional kernels of the operators on the two sides of the equation have the same dimension. In order to show this we will construct maps between the kernels of $D_\A$ and $\bigsqcup_{\mu,i}D_{\A^i_\mu}$ and show they are injective when $\epsilon$ is small enough. Rescaling back to a fix $\epsilon$ then implies the result for any $\epsilon > 0$.

Fix cut-off functions $\gamma^i_\mu, \beta^i_\mu$ with
\[
\gamma^i_\mu \equiv \begin{dcases} 1 & \mbox{ on }B_1(p^i_\mu)\times S^1,\\ 0  	& \mbox{ outside }B_2(p^i_\mu)\times S^1, \end{dcases} \qquad \beta^i_\mu \equiv \begin{dcases} 0 & \mbox{ on }B_1(p^i_\mu)\times S^1,\\ 1 & \mbox{ outside }B_2(p^i_\mu)\times S^1, \end{dcases}
\]
and the additional constraint that $\beta^i_\mu\equiv 1 $ on the support of $d\gamma^i_\mu$.

Now, suppose that $\xi$ is an element in the kernel of $D_\A$. We define an element $\xi^i_\mu$ in the kernel of $D_{\A^i_\mu}$ by $\xi^i_\mu = \gamma^i_\mu\xi - G^i_\mu D_{\A^i_\mu} (\gamma^i_\mu\xi)$. Indeed, note that
\begin{equation}\label{eq:Dimension:1}
\| D^i_\mu (\gamma^i_\mu\xi)\|_{C^{0,\alpha}_{\delta-2}} = \| d\gamma^i_\mu \otimes \xi\|_{C^{0,\alpha}_{\delta-2}} \lesssim \| \beta^i_\mu \xi\|_{C^{0,\alpha}_{\delta-1}}<\infty
\end{equation}
since $D_{\A}=D_{\A^i_\mu}$ on the support of $\gamma^i_\mu$ and on the support of $d\gamma^i_\mu$ the weight function $r^i_\mu =O(1)$ and $\beta^i_\mu \equiv 1$. Here $a\lesssim b$ means $a\leq C b$ for a constant $C>0$ independent of $a,b$ and $\epsilon$.

We claim that the map $\xi \mapsto \{ \xi^i_\mu\}_{\mu, i}$ is injective for $\epsilon$ sufficiently small. Suppose not, so that there exists $\xi\in \ker D_{\A}$ such that $\gamma^i_\mu\xi = G^i_\mu D^i_\mu (\gamma^i_\mu\xi)$ for all $\mu, i$. Propositions \ref{prop:Approximate:Dirac:Kernel} and \ref{prop:Approximate:Dirac:Schauder} imply that for $\epsilon$ small enough we have
\begin{equation}\label{eq:Dimension:2}
\| \xi \|_{C^{1,\alpha}_{\delta-1}} \lesssim \sum_{\mu,i}\| \gamma^i_\mu \xi\|_{C^{0,\alpha}_{\delta-1}}
\end{equation}

Then, using Propositions \ref{prop:Linearised} and \ref{prop:Approximate:Dirac:Kernel} and \eqref{eq:Dimension:1} we obtain 
\[
\| \gamma^i_\mu\xi \|_{C^{0,\alpha}_{\delta-1}} \leq \| \gamma^i_\mu\xi \|_{C^{1,\alpha}_{\delta-1}} = \| G^i_\mu D^i_\mu (\gamma^i_\mu\xi) \|_{C^{1,\alpha}_{\delta-1}} \lesssim \| \beta^i_\mu \xi\|_{C^{0,\alpha}_{\delta-1}} \leq \eta \|  \xi \|_{C^{1,\alpha}_{\delta-1}}.
\]
Summing over $\mu, i$ and using \eqref{eq:Dimension:2} implies that for $\eta$, and therefore $\epsilon$, small enough we must in fact have $\xi=0$.

Hence we have shown that ${\rm dim}\ker D_{\A} \leq \sum_{\mu,i}{\rm dim}\ker D_{\A^i_\mu}$. The opposite inequality is shown in analogous way by constructing an injective map from $\bigoplus_{\mu,i}\ker D_{\A^i_\mu}$ to $\ker D_{\A}$.
\endproof

\bibliographystyle{plain}
\bibliography{calorons_Lorenzo}

\begin{thebibliography}{10}

\bibitem{Braverman:2016}
Alexander Braverman, Michael Finkelberg, and Hiraku Nakajima.
\newblock Coulomb branches of {$3d$} {$\mathcal{N}=4$} quiver gauge theories
  and slices in the affine {G}rassmannian.
\newblock {\em Adv. Theor. Math. Phys.}, 23(1):75--166, 2019.

\bibitem{Bruckmann:2002}
Falk Bruckmann and Pierre van Baal.
\newblock Multi-caloron solutions.
\newblock {\em Nuclear Phys. B}, 645(1-2):105--133, 2002.

\bibitem{CH_2007}
Benoit Charbonneau and Jacques Hurtubise.
\newblock The {N}ahm transform for calorons.
\newblock In {\em The many facets of geometry}, pages 34--70. Oxford Univ.
  Press, Oxford, 2010.

\bibitem{Cherkis_2021}
Sergey~A. Cherkis, Andr\'{e}s Larra\'{\i}n-Hubach, and Mark Stern.
\newblock Instantons on multi-{T}aub-{NUT} spaces {I}: asymptotic form and
  index theorem.
\newblock {\em J. Differential Geom.}, 119(1):1--72, 2021.

\bibitem{Cherkis:Stern:II}
Sergey~A. Cherkis, Andr\'es Larra\'in-Hubach, and Mark Stern.
\newblock Instantons on multi-{T}aub-{NUT} {S}paces {II}: {B}ow construction,
  2021.
\newblock https://arxiv.org/abs/2103.12754.

\bibitem{Cork_thesis}
Josh Cork.
\newblock {\em Calorons, symmetry, and the soliton trinity}.
\newblock PhD thesis, University of Leeds, 2018.

\bibitem{Cork:2017}
Josh Cork.
\newblock Symmetric calorons and the rotation map.
\newblock {\em J. Math. Phys.}, 59(6):062902, 23, 2018.
\newblock [Addendum: J.Math.Phys. 59, 079901 (2018)].

\bibitem{Donaldson:Kronheimer}
Simon.~K. Donaldson and Peter.~B. Kronheimer.
\newblock {\em The geometry of four-manifolds}.
\newblock Oxford Mathematical Monographs. The Clarendon Press Oxford University
  Press, New York, 1990.
\newblock Oxford Science Publications.

\bibitem{Folland}
Gerald~B. Folland.
\newblock Harmonic analysis of the de {R}ham complex on the sphere.
\newblock {\em J. Reine Angew. Math.}, 398:130--143, 1989.

\bibitem{Foscolo:Deformation}
Lorenzo Foscolo.
\newblock Deformation theory of periodic monopoles (with singularities).
\newblock {\em Comm. Math. Phys.}, 341(1):351--390, 2016.

\bibitem{FHN}
Lorenzo Foscolo, Mark Haskins, and Johannes Nordstr\"{o}m.
\newblock Complete noncompact {$G_2$}-manifolds from asymptotically conical
  {C}alabi-{Y}au 3-folds.
\newblock {\em Duke Math. J.}, 170(15):3323--3416, 2021.

\bibitem{Freed:Uhlenbeck}
Daniel~S. Freed and Karen~K. Uhlenbeck.
\newblock {\em Instantons and four-manifolds}, volume~1 of {\em Mathematical
  Sciences Research Institute Publications}.
\newblock Springer-Verlag, New York, 1984.

\bibitem{Harland_2007}
Derek Harland.
\newblock Large scale and large period limits of symmetric calorons.
\newblock {\em J. Math. Phys.}, 48(8):082905, 21, 2007.

\bibitem{HS:1978}
Barry~J. Harrington and Harvey~K. Shepard.
\newblock Periodic euclidean solutions and the finite-temperature
  {Y}ang-{M}ills gas.
\newblock {\em Phys. Rev. D}, 17:2122--2125, 1978.

\bibitem{HHM}
Tam\'{a}s Hausel, Eugenie Hunsicker, and Rafe Mazzeo.
\newblock Hodge cohomology of gravitational instantons.
\newblock {\em Duke Math. J.}, 122(3):485--548, 2004.

\bibitem{Jaffe:Taubes}
Arthur Jaffe and Clifford~H. Taubes.
\newblock {\em Vortices and monopoles}, volume~2 of {\em Progress in Physics}.
\newblock Birkh\"auser Boston, 1980.
\newblock Structure of static gauge theories.

\bibitem{Jarvis:Rational:Maps:II}
Stuart Jarvis.
\newblock Construction of {E}uclidean monopoles.
\newblock {\em Proc. London Math. Soc. (3)}, 77(1):193--214, 1998.

\bibitem{Jarvis:Rational:Maps:I}
Stuart Jarvis.
\newblock Euclidean monopoles and rational maps.
\newblock {\em Proc. London Math. Soc. (3)}, 77(1):170--192, 1998.

\bibitem{Kato:2020}
Takumi Kato, Atsushi Nakamula, and Koki Takesue.
\newblock Symmetric calorons of higher charges and their large period limits.
\newblock {\em J. Geom. Phys.}, 162:Paper No. 104071, 20, 2021.

\bibitem{Kraan:1998sn}
Thomas~C. Kraan and Pierre van Baal.
\newblock {Monopole constituents inside SU(n) calorons}.
\newblock {\em Phys. Lett. B}, 435:389--395, 1998.

\bibitem{Kraan_1998}
Thomas~C. Kraan and Pierre van Baal.
\newblock Periodic instantons with non-trivial holonomy.
\newblock {\em Nuclear Phys. B}, 533(1-3):627--659, 1998.

\bibitem{Kraan:1998gh}
Thomas~C. Kraan and Pierre van Baal.
\newblock {Constituent monopoles without gauge fixing}.
\newblock {\em Nucl. Phys. B Proc. Suppl.}, 73:554--556, 1999.

\bibitem{Lee2_1998}
Kimyeong Lee.
\newblock Instantons and magnetic monopoles on {${\bf R}^3\times S^1$} with
  arbitrary simple gauge groups.
\newblock {\em Phys. Lett. B}, 426(3-4):323--328, 1998.

\bibitem{Lee_1998}
Kimyeong Lee and Changhai Lu.
\newblock {$SU(2)$} calorons and magnetic monopoles.
\newblock {\em Phys. Rev. D}, 58(2), 1998.

\bibitem{Lee:2000hp}
Kimyeong Lee, David Tong, and Sangheon Yi.
\newblock Moduli space of two {${\rm U}(1)$} instantons on noncommutative
  {$\bold R^4$} and {$\bold R^3\times S^1$}.
\newblock {\em Phys. Rev. D (3)}, 63(6):065017, 10, 2001.

\bibitem{IndexFibredBoundary}
Eric Leichtnam, Rafe Mazzeo, and Paolo Piazza.
\newblock The index of {D}irac operators on manifolds with fibered boundaries.
\newblock {\em Bull. Belg. Math. Soc. Simon Stevin}, 13(5):845--855, 2006.

\bibitem{Nakajima:Takayama}
Hiraku Nakajima and Yuuya Takayama.
\newblock Cherkis bow varieties and {C}oulomb branches of quiver gauge theories
  of affine type {$A$}.
\newblock {\em Selecta Math. (N.S.)}, 23(4):2553--2633, 2017.

\bibitem{Norbury:Loop:Group}
Paul Norbury.
\newblock Periodic instantons and the loop group.
\newblock {\em Comm. Math. Phys.}, 212(3):557--569, 2000.

\bibitem{Nye_thesis}
Tom M.~W. Nye.
\newblock {\em The Geometry of calorons}.
\newblock PhD thesis, University of Edinburgh, 2001.

\bibitem{Nye:Singer}
Tom M.~W. Nye and Michael~A. Singer.
\newblock An {$L^2$}-index theorem for {D}irac operators on {$S^1\times\bold
  R^3$}.
\newblock {\em J. Funct. Anal.}, 177(1):203--218, 2000.

\bibitem{Panyushev_2015}
Dmitri~I. Panyushev.
\newblock The {D}ynkin index and {$\mathfrak{sl}_2$}-subalgebras of simple
  {L}ie algebras.
\newblock {\em J. Algebra}, 430:15--25, 2015.

\bibitem{Takayama_2016}
Yuuya Takayama.
\newblock Nahm's equations, quiver varieties and parabolic sheaves.
\newblock {\em Publ. Res. Inst. Math. Sci.}, 52(1):1--41, 2016.

\bibitem{Taubes:G:Monopoles}
Clifford~H. Taubes.
\newblock The existence of multimonopole solutions to the nonabelian,
  {Y}ang-{M}ills-{H}iggs equations for arbitrary simple gauge groups.
\newblock {\em Comm. Math. Phys.}, 80(3):343--367, 1981.

\bibitem{YWang}
Yuanqi Wang.
\newblock Moduli spaces of {$G_2$}-instantons and {$Spin(7)$}-instantons on
  product manifolds.
\newblock {\em Ann. Henri Poincar\'{e}}, 21(9):2997--3033, 2020.

\bibitem{Ward:2003sx}
Richard~S. Ward.
\newblock Symmetric calorons.
\newblock {\em Phys. Lett. B}, 582(3-4):203--210, 2004.

\end{thebibliography}
\end{document}